\documentclass[11pt]{amsart}
\usepackage{amssymb,amsmath,color}
\usepackage{tikz-dependency}
\usepackage{tikz}
\usepackage{enumerate}

\newtheorem{theorem}{Theorem}[section]
\newtheorem{conjecture}[theorem]{Conjecture}
\newtheorem{proposition}[theorem]{Proposition}
\newtheorem{lemma}[theorem]{Lemma}
\newtheorem{definition}[theorem]{Definition}

\newtheorem{corollary}[theorem]{Corollary}
\newtheorem{example}[theorem]{Example}
\newtheorem{remark}[theorem]{Remark}

\numberwithin{equation}{section}

\newcommand{\Qqq}{{\mathbb Q}}
\newcommand{\Zzz}{{\mathbb Z}}

\newcommand{\hz}{\hat{0}}
\newcommand{\ho}{\hat{1}}

\newcommand{\tensor}{\otimes}

\newcommand{\ab}{\av\bv}
\newcommand{\ba}{\bv\av}
\newcommand{\av}{{\bf a}}
\newcommand{\bv}{{\bf b}}
\newcommand{\cd}{\cv\dv}
\newcommand{\ctd}{\cv\mbox{-}2\dv}
\newcommand{\cv}{{\bf c}}
\newcommand{\dv}{{\bf d}}

\newcommand{\zab}{\Zzz\langle\av,\bv\rangle}
\newcommand{\zcd}{\Zzz\langle\cv,\dv\rangle}

\newcommand{\zx}{\Zzz[x]}
\newcommand{\Rab}{R\langle\av,\bv\rangle}
\newcommand{\Rcd}{R\langle\cv,\dv\rangle}
\newcommand{\Sab}{S\langle\av,\bv\rangle}
\newcommand{\Scd}{S\langle\cv,\dv\rangle}

\DeclareMathOperator{\tail}{tail}
\DeclareMathOperator{\head}{head}
\DeclareMathOperator{\id}{id}

\DeclareMathOperator{\Asc}{Asc}
\DeclareMathOperator{\Des}{Des}
\DeclareMathOperator{\QSym}{QSym}

\newcommand{\A}{A}
\newcommand{\B}{B}

\newcommand{\cupdot}{\mathbin{\dot{\cup}}}

\newcommand{\vanish}[1]{}

\begin{document}

\title{Balanced and Bruhat graphs}

\author{Richard EHRENBORG and Margaret READDY}

\address{Department of Mathematics, University of Kentucky, Lexington,
  KY 40506-0027.\hfill\break \tt http://www.math.uky.edu/\~{}jrge/,
  richard.ehrenborg@uky.edu.}

\address{Department of Mathematics, University of Kentucky, Lexington,
  KY 40506-0027.\hfill\break
\tt http://www.math.uky.edu/\~{}readdy/,
margaret.readdy@uky.edu.}

\date{\today}

\begin{abstract}
We generalize chain
enumeration in graded partially ordered sets
by relaxing the graded, poset and Eulerian requirements.
The resulting balanced digraphs,
which include 
the classical Eulerian posets having an $R$-labeling,
imply the existence of the (non-homogeneous)
$\cd$-index,
a key invariant for studying inequalities for
the flag vector  of polytopes.
Mirroring Alexander duality for Eulerian posets,
we show an analogue of Alexander duality
for bounded balanced digraphs.
For Bruhat graphs of Coxeter groups,
an important family 
of balanced graphs,
our theory gives elementary proofs
of the existence of the complete $\cd$-index
and its properties.
We also introduce 
the rising and falling quasisymmetric functions
of a labeled acyclic digraph and show they
are Hopf algebra homomorphisms
mapping
balanced digraphs to the Stembridge peak algebra.
We conjecture non-negativity of the $\cd$-index for acyclic digraphs
having a balanced linear edge labeling.
\end{abstract}

\maketitle

\section{Introduction}

The $\cd$-index is an important invariant for studying
face incidence data of polytopes,
and more generally, 
chain enumeration of Eulerian posets.
It is a non-commutative polynomial which
removes all 
the linear redundancies which hold among the 
flag vector entries~\cite{Bayer_Klapper}
as described by the 
generalized Dehn--Sommerville relations~\cite{Bayer_Billera}.
Ehrenborg and Readdy's  discovery of the  inherent 
coalgebraic structure of the $\cd$-index
and the techniques
developed in~\cite{Ehrenborg_Readdy_c} have
been applied 
to settle many fundamental problems,
including 
giving compact proofs of 
old results~\cite{Bayer_Billera,Billera_Ehrenborg_Readdy_z},
transparent techniques to compute
flag vectors of oriented matroids~\cite{Billera_Ehrenborg_Readdy_om},
explicit formulas for the toric $h$-vector,
versions of 
Stanley's Gorenstein* 
conjecture~\cite{Billera_Ehrenborg,Billera_Ehrenborg_Readdy_z} 
leading up to a proof of the conjecture itself~\cite{Ehrenborg_Karu},
new non-trivial inequalities among the face
incidence data of polytopes~\cite{Ehrenborg_lifting,Ehrenborg_zonotopes}
and 
extending 
classical subspace arrangement results
to other manifolds~\cite{Ehrenborg_Goresky_Readdy,Ehrenborg_Readdy_Slone}.

There are two new developments in this area. The first is
work of Ehrenborg, Goresky and Readdy,
who extend flag vector
enumeration ideas to Whitney stratified spaces and quasi-graded
posets~\cite{Ehrenborg_Goresky_Readdy,Ehrenborg_Readdy_manifolds}. 
The very notion of enumeration is replaced with
the topologically meaningful Euler-enumeration in the case of Whitney
stratified-spaces, and weighted zeta functions in the case of
quasi-graded posets. The Eulerian condition becomes a natural
condition involving the Euler characteristic and weighted zeta
function, respectively. Unlike the case of polytopes and regular
decompositions of spheres, the coefficients of the $\cd$-index 
can be
negative, expanding the nature of questions in the
field.

The second development is Billera and Brenti's work on the
``complete'' $\cd$-index, a nonhomogeneous extension of the
$\cd$-index~\cite{Billera_Brenti}.
It is known that the (strong) Bruhat order on a
Coxeter group forms an Eulerian poset~\cite{Verma},
hence any interval has a $\cd$-index.
Using quasisymmetric function theory,
they prove the Bruhat graph,
a directed graph which includes the cover relations of the Bruhat
order as well as ``algebraic shortcuts'' between elements,
 has a
{\em non-homogeneous} $\cd$-index.
Furthermore, they show  one can compute the Kazhdan--Lusztig polynomials 
via
this complete $\cd$-index of Bruhat
intervals.
It is exactly this paper which motivated the 
present authors  to
look for a general setting to guarantee the
existence of this non-homogeneous $\cd$-index.

Recall a partially ordered set (poset) is {\em graded} 
if its elements have a
well-defined distance from the minimal element of the poset.
Bj\"orner and Stanley~\cite[Theorem~2.7]{Bjorner}
showed that 
if a graded poset has a combinatorial
labeling of its cover relations known as an $R$-labeling,
one can determine the flag $f$-vector in terms of the labeling
inherited by the maximal chains.
When the
poset is Eulerian,
that is, every interval satisfies the Euler--Poincar\'e relation,
one can reduce this information
to the classical $\cd$-index.

By relaxing the 
graded, poset and Eulerian requirements,
we study
a general class of labeled directed graphs which satisfy a 
{\em balanced condition}.
Recall a poset having an $R$-labeling demands that there be 
exactly one rising chain
in each interval of the poset
and,
if the poset is Eulerian, exactly
one falling chain in each interval.
Our balanced condition states the number of rising paths of length~$k$
must equal
the number of falling paths of length $k$.
This allows us to directly prove the existence of the $\cd$-index
for balanced graphs 
and capture the results for
Bruhat graphs as an important special case.

The presentation we give is self-contained.  
To underscore the connection with posets,
results which also hold for the $\ab$- and $\cd$-index of graded posets
will be stated as separate remarks.

An overview of the paper is as follows.
In Section~\ref{section_labeled_graphs}
we introduce
the notion of a labeled acyclic digraph
in order to model poset structure in this more general
setting.   
An interpretation of its chain enumeration is given in terms
of directed paths in the graph.
In Section~\ref{section_coalgebras}
we then set the coalgebraic
groundwork for flag enumeration
in labeled acyclic digraphs.
We show the $\ab$-index of a labeled acyclic digraph is
a coalgebra homomorphism from the linear span of bounded 
labeled acyclic digraphs to the polynomial ring
$\Zzz\langle \av, \bv \rangle$; 
see Corollary~\ref{corollary_coalgebra_homomorphism}.

We introduce the $\tilde{r}$ and $\tilde{f}$
polynomials in Section~\ref{section_the_cd_index}
to $q$-enumerate
the rising and falling chains in 
the intervals of a labeled digraph.
These polynomials hark back to 
the theory of Coxeter groups and Kazhdan--Lusztig polynomials;
see~\cite[Chapter~5]{Bjorner_Brenti}
as well as~\cite{Kazhdan_Lusztig_Hecke,Kazhdan_Lusztig_Schubert}.
The main result
(Theorem~\ref{theorem_cd})
gives three equivalent statements
which imply the (non-homogeneous) $\ab$-index 
of an acyclic digraph can be written as a (non-homogeneous)
$\cd$-index.
The key condition is that the
number of rising paths of length $k$ equals
the number of falling paths of length $k$.
We include a second  proof of one of the implications
in Theorem~\ref{theorem_cd}
which uses  Hochschild cohomology.

A theory that mimics the notion of an Eulerian poset
would not be complete without Alexander duality.
Recall that for an Eulerian poset $P$ with
decomposition $S \cupdot T \cupdot \{\hz,\ho\}$
and rank function~$\rho$,
the celebrated Alexander duality states that
the M\"obius function values
$\mu(\hz,\ho)$ of each of the two posets
$S \cup \{\hz,\ho\}$ and $T \cup \{\hz,\ho\}$
are equal up to the sign $(-1)^{\rho(P)-1}$.
In Section~\ref{section_Alexander}
we introduce the notion of
a restricted digraph. 
We state Alexander duality
where the M\"obius function is
replaced by a signed sum over falling chains.

In Section~\ref{section_Bruhat} we apply our results
to the important family of Bruhat graphs.
Using the existence of
a reflection ordering, introduced by Dyer~\cite{Dyer},
the existence of the $\cd$-index of the Bruhat graph and its
properties follow.

In Section~\ref{section_quasisymmetric_functions}
we review the basic set-up surrounding the 
ring of quasisymmetric functions.
For a bounded labeled digraph we introduce the  {\em rising} and
{\em falling quasisymmetric functions} and relate these
with a shift of the aforementioned rising and falling
polynomials.
We show the rising
and falling quasisymmetric functions  are Hopf algebra homomorphisms
from the Hopf algebra formed by the linear span of bounded
labeled acyclic digraphs
to the quasisymmetric functions.
We reformulate Theorem~\ref{theorem_cd} in terms of
Stembridge's peak algebra~\cite{Stembridge}.

Section~\ref{section_linear}
begins with the result that 
given any polynomial in 
the ring $\zcd$
having non-negative coefficients, we show how
to build an Eulerian poset having this $\cd$-polynomial.

Recall that the classical $\cd$-index of 
the face lattice of a polytope,
and more generally,
any
spherically-shellable poset,
has non-negative coefficients~\cite{Stanley_d}.
Non-negativity also holds for Gorenstein* posets~\cite{Karu}.
These results form  two cornerstones for the research program
of classifying all the linear inequalities satisfied by the
$\cd$-index.
We conjecture non-negativity
for the $\cd$-index of a bounded labeled acyclic digraph equipped
with a balanced edge labeling that is linear;
see Conjecture~\ref{conjecture_linear}.

In the concluding remarks
we end with open questions and research
directions we are pursuing.

\section{Labeled graphs}
\label{section_labeled_graphs}

We begin by introducing a class of directed graphs
in order to relax
the notion of grading in a graded partially ordered set (poset).
For further details about posets, see~\cite[Chapter~3]{Stanley_EC_I}.

Let $G = (V,E)$ be a directed, acyclic and locally finite graph
with
multiple edges allowed.
Recall that an {\em acyclic graph} does not have any directed cycles
and 
the property of a graph being {\em locally finite}
requires that 
there are a finite number
of paths between any two vertices.
Each directed edge~$e$ has a tail and a head,
denoted respectively by
$\tail(e)$ and $\head(e)$.
View each directed edge as an arrow from its tail to its head.
A directed path $p$ of length $k$ from
a vertex $x$ to a vertex $y$ is a list
of $k$ directed edges $(e_{1}, e_{2}, \ldots, e_{k})$ such that
$\tail(e_{1}) = x$,
$\head(e_{k}) = y$ and
$\head(e_{i}) = \tail(e_{i+1})$ for $i = 1, \ldots, k-1$.
We denote the length of a path $p$
by $\ell(p)$.

Since the graph is acyclic, it does not have any loops.
Furthermore, the acyclicity condition implies
there is a natural partial order on the vertices of $G$ by defining
the order relation $x \leq y$ if there is a directed path
from the vertex $x$ to the vertex $y$. It is straightforward to
verify that this relation is
reflexive, antisymmetric and transitive.
Furthermore, it allows us to define the {\em interval}
$[x,y]$ to be 
$$
     [x,y] = \{z \in V(G) \:\: : \:\:
                  \text{there is a directed path from $x$ to $z$ 
                        and a directed path from $z$ to $y$}\}.
$$
We view the interval $[x,y]$ as the vertex-induced
subgraph of the digraph $G$, where the edges have the
same labels as in the digraph $G$.
The locally finite condition is now
equivalent to that 
every interval~$[x,y]$
in the graph has finite cardinality.

\begin{example}
{\rm
Consider
a (locally finite) poset~$P$ and let the directed edges be the
cover relations of the poset, in other words, the Hasse
diagram of $P$ is the digraph.  When we draw the Hasse
diagram of a poset we view its edges as being directed upward.
Moreover, the fact the poset is locally finite
implies that the associated digraph is locally finite.
Hence this is an acyclic digraph.
}
\end{example}

A relaxed notion of edge labeling is needed which will enable
us to define the $\ab$-index,
and ultimately,
the $\cd$-index.
Let $\Lambda$ be a set with a relation $\sim$,
that is, there is a subset $R \subseteq \Lambda \times \Lambda$
such that for $i,j \in \Lambda$ we have
$i \sim j$ if and only if $(i,j) \in R$.
A {\em labeling of $G$} is a function $\lambda$ from the set of edges
of $G$ to the set~$\Lambda$. 
Let $\av$ and $\bv$ be two non-commutative variables
each of degree one. 
For a path $p= (e_{1}, \ldots, e_{k})$ of length $k$, 
where $k \geq 1$, we define
the {\em descent word} $u(p)$ to be the $\ab$-monomial
$u(p) = u_{1} u_{2} \cdots u_{k-1}$, where
$$    u_{i}
   =
\begin{cases}
    \av & \text{ if } \lambda(e_{i}) \sim     \lambda(e_{i+1}) , \\
    \bv & \text{ if } \lambda(e_{i}) \not\sim \lambda(e_{i+1}) .
\end{cases} $$
Observe that the descent word $u(p)$
has degree $k-1$, that is,
one less than the length of the path $p$.
The {\em $\ab$-index} of an interval $[x,y]$ 
is defined to be
\begin{equation}
    \Psi([x,y]) 
       =
    \sum_{p} u(p)  , 
\label{equation_ab-index}
\end{equation}
where the sum is over all directed paths $p$ from $x$ to $y$.

\begin{example}
{\rm 
In the case when the relation on $\Lambda$ is
a linear order, the digraph is the
Hasse diagram of a graded poset and every interval
has a unique rising chain. This condition is the classical
notion of $R$-labeling introduced by
Bj\"orner and Stanley~\cite{Bjorner}.
}
\end{example}

In keeping with the poset motivation,
we will continue to use the terminology
rising and falling in our more general setting.
See the paper~\cite{Bjorner_Wachs},
where
Bj\"orner and Wachs
weakened the condition that $\Lambda$ is a linear order
to a partial order.
As a remark, one can further loosen the condition on the relation on
$\Lambda$ so that the only labels which need to be compared
are pairs of elements $(\lambda(e), \lambda(f))$
such that $\head(e) = \tail(f)$.

For graded posets with an $R$-labeling
equation~\eqref{equation_ab-index}
gives a different definition of the notion
of the $\ab$-index of a poset.
See~\cite[Lemma~3.1]{Ehrenborg_Readdy_r} for
more details.

Given a labeled directed graph $G$, define
the graph $G^{*}$ by reversing all the edges,
keeping the edge labeling the same,
and reversing the relation $\sim$ on $\Lambda$,
that is, for $e \in E(G)$ we have
$\head_{G^{*}}(e) = \tail_{G}(e)$
and
$\tail_{G^{*}}(e) = \head_{G}(e)$.
The labeling is given by $\lambda_{G^{*}}(e) = \lambda_{G}(e)$.
Finally, the new relation $\Lambda^{*}$ is given by
$i \sim^{*} j$ if and only if $j \sim i$
for $i,j \in \Lambda$.
For an $\ab$-monomial
$u = u_{1} u_{2} \cdots u_{k}$
define the reverse monomial by
$u^{*} = u_{k} \cdots u_{2} u_{1}$
and extend linearly to an involution on 
the non-commutative polynomial ring $\zab$.
Observe that a path~$p$ from~$x$ to~$y$ in~$G$
corresponds to a path~$p^{*}$ from~$y$ to~$x$ in~$G^{*}$.
Moreover, the descent word of the path satisfies
$u(p^{*}) = u(p)^{*}$. Finally this relation extends to the
$\ab$-index of the entire interval~$[x,y]$, that is,
$\Psi([x,y]^{*}) = \Psi([y,x]) = \Psi([x,y])^{*}$.

\section{Coalgebras}
\label{section_coalgebras}

In this section we develop the
underlying coalgebraic structure
of labeled acyclic digraphs.

Let $\zab$ be the non-commutative polynomial ring
in the degree $1$ variables $\av$ and $\bv$ with integer
coefficients.
On the ring $\zab$ define a coproduct $\Delta$
by defining it on an $\ab$-monomial $u_{1} u_{2} \cdots u_{n}$
by
$$   \Delta(u_{1} u_{2} \cdots u_{n})
   =
     \sum_{i=1}^{n} u_{1} \cdots u_{i-1}  \tensor  u_{i+1} \cdots u_{n} ,
$$
and extend by linearity to $\zab$.
This coproduct, together with the usual multiplication,
does not form a bialgebra. Instead the
{\em Newtonian condition} is satisfied:
\begin{equation}
\Delta(v \cdot w)
   =
\sum_{w}  v \cdot w_{(1)} \tensor w_{(2)}
   +
\sum_{v}  v_{(1)} \tensor v_{(2)} \cdot w  .
\label{equation_Newtonian_condition}
\end{equation}
Here we use the Sweedler notation for
the coproduct~\cite{Joni_Rota,Sweedler}.
This gives the ring $\zab$ a Newtonian coalgebra structure.

\begin{theorem}
For a labeled acyclic digraph $G$
with two vertices $x$ and $y$,
the following holds:
$$   \Delta(\Psi([x,y]))
   =
     \sum_{x < z < y}
            \Psi([x,z]) \tensor \Psi([z,y])   .  
$$
\label{theorem_coalgebra_homomorphism}
\end{theorem}
\begin{proof}
For a path $p = (e_{1}, \ldots, e_{k})$
let $i(p)$ denote the set of interior
vertices on the path, that is,
$i(p) = \{\head(e_{1}), \ldots, \head(e_{k-1})\}$.
Furthermore, for a path $p$ from $x$ to $y$ and $x \leq z < w \leq y$,
let $p|_{[z,w]}$ denote the path restricted to the interval $[z,w]$.
Now we have
\begin{align*}
     \Delta(\Psi([x,y]))
& =
     \sum_{p} \Delta(u(p)) \\
& =
     \sum_{p}
       \sum_{i=1}^{\ell(p)-1} 
           u_{1}(p) \cdots u_{i-1}(p)
         \tensor
           u_{i+1}(p) \cdots u_{\ell(p)-1}(p) \\
& =
     \sum_{p}
       \sum_{z \in i(p)}
           u(p|_{[x,z]}) \tensor u(p|_{[z,y]}) \\
& =
     \sum_{x < z < y}
       \sum_{p : z \in i(p)}
           u(p|_{[x,z]}) \tensor u(p|_{[z,y]}) \\
& =
     \sum_{x < z < y}
         \left(
           \sum_{p_{1}} u(p_{1})
         \right)
       \tensor
         \left(
           \sum_{p_{2}} u(p_{2})
         \right) \\
& =
     \sum_{x < z < y}
         \Psi([x,z])
       \tensor
         \Psi([z,y])   .
\end{align*}
Here we are summing over all maximal paths $p$ in the
interval $[x,y]$.  
In the second to last equality
$p_{1}$ and $p_{2}$ are paths in $[x,z]$,
respectively $[z,y]$. 
\end{proof}

\begin{remark}
{\rm
In the case of Bruhat graphs,
Theorem~\ref{theorem_coalgebra_homomorphism}
was stated in~\cite[Proposition~2.11]{Billera_Brenti}. 
}
\end{remark}

An acyclic digraph $G$ is {\em bounded} if has
a unique source and a unique sink.
Following poset notation, we denote the unique source by
$\hz$ and the unique sink by $\ho$.
For brevity,
we let $\Psi(G)$ denote $\Psi([\hz,\ho])$.

For two bounded labeled acyclic digraphs $G$ and $H$
we define the product $G * H$ as follows.
We tacitly assume that $V(G), V(H), E(G), E(H), \Lambda_{G}$
and $\Lambda_{H}$ are disjoint.
Let the vertex set of $G * H$ be the disjoint union of 
$V(G) - \{\ho\}$ and $V(H) - \{\hz\}$,
that is,
$V(G * H) = (V(G) - \{\ho_{G}\}) \cup (V(H) - \{\hz_{H}\})$.
Let the edge set be
\begin{align*}
     E(G * H)
& =
     \{e \in E(G) \:\: : \:\: \head(e) \neq \ho_{G}\} \\
&  \cup \:\: 
     \{f \in E(H) \:\: : \:\: \tail(f) \neq \hz_{H}\} \\
&  \cup \:\: 
     \{(e,f) \in E(G) \times E(H) \:\: : \:\:
          \head(e) = \ho_{G}, \tail(f) = \hz_{H} \} ,
\end{align*}
where the new edge $(e,f)$ is
defined by $\tail((e,f)) = \tail(e)$ and $\head((e,f)) = \head(f)$.
Let the label set $\Lambda$ be defined by
$ 
   \Lambda
 =
   \Lambda_{G} \cup \Lambda_{H} \cup \Lambda_{G} \times \Lambda_{H}$,
with the relation on $\Lambda$ given
by the following four cases:
$$
\begin{cases}
  \lambda \sim \mu
& \text{ if }
 \lambda, \mu \in \Lambda_{G}, \lambda \sim_{\Lambda_{G}} \mu, \\
  \lambda \sim (\mu_{1}, \mu_{2})
& \text{ if }
 \lambda, \mu_{1} \in \Lambda_{G},
           \mu_{2} \in \Lambda_{H},
  \lambda \sim_{\Lambda_{G}} \mu_{1}, \\
  (\lambda_{1}, \lambda_{2}) \sim \mu
& \text{ if }
 \lambda_{1}      \in \Lambda_{G},
  \lambda_{2}, \mu \in \Lambda_{H},
  \lambda_{2} \sim_{\Lambda_{H}} \mu, \\
  \lambda \sim \mu
& \text{ if }
 \lambda, \mu \in \Lambda_{H}, \lambda \sim_{\Lambda_{H}} \mu.
\end{cases}
$$
Finally, define the labeling $\lambda : E(G*H) \longrightarrow \Lambda$
by the three cases
$$  \begin{cases}
         \lambda(e) = \lambda_{G}(e) 
         & \text{ if } e \in E(G) , \\
         \lambda((e,f)) = (\lambda_{G}(e), \lambda_{H}(f))
         & \text{ if } (e,f) \in E(G) \times E(H) , \\
         \lambda(f) = \lambda_{H}(f) 
         & \text{ if } f \in E(H) .
           \end{cases}
$$
This product is the labeled analogue of the Stanley product of posets;
see~\cite{Stanley_d}.

\begin{theorem}
Let $G$ and $H$ be two bounded labeled acyclic digraphs,
where each has a unique source and unique sink.
Then
the $\ab$-index satisfies
$$  \Psi(G * H) = \Psi(G) \cdot \Psi(H)   ,   $$
where $*$ is the labeled analogue of the Stanley product
of posets.
\label{theorem_algebra_homomorphism}
\end{theorem}
\begin{proof}
Each directed path $p$ from $\hz$ to $\ho$
in $G*H$ has the form
$p = (e_{1}, \ldots, e_{i-1}, (e_{i},f_{1}), f_{2}, \ldots, f_{j})$,
which factors into the two paths
$p_{1} = (e_{1}, \ldots, e_{i-1}, e_{i})$
and
$p_{2} = (f_{1}, f_{2}, \ldots, f_{j})$
in $G$, respectively~$H$.
The descent word also factors as
$u(p) = u(p_{1}) \cdot u(p_{2})$.
By summing over all paths, the result follows.
\end{proof}

Let $\mathcal{G}$ be the linear span of 
bounded labeled acyclic digraphs with $\hz \neq \ho$.
The space $\mathcal{G}$ is a Newtonian coalgebra
with the product $*$ and the coproduct
$$  \Delta(G) = \sum_{\hz < z < \ho} [\hz,z] \tensor [z,\ho] . $$
Theorems~\ref{theorem_coalgebra_homomorphism}
and~\ref{theorem_algebra_homomorphism}
imply the following corollary.
\begin{corollary}
\label{corollary_coalgebra_homomorphism}
The $\ab$-index is a coalgebra homomorphism from 
$\mathcal{G}$,  the linear span of 
bounded labeled acyclic digraphs with $\hz \neq \ho$,
to the coalgebra $\zab$.
\end{corollary}

On the coalgebra $\zab$ define an involution
$u \longmapsto \overline{u}$ by uniformly exchanging $\av$'s
and $\bv$'s. Observe this involution is a Newtonian
coalgebra automorphism, that is, the product
and the coproduct satisfy
$$     \overline{u \cdot v}
    =
       \overline{u} \cdot \overline{v} 
\:\:\:\: \text{ and } \:\:\:\:
       \Delta(\overline{u})
    =
       \sum_{u} \overline{u_{(1)}} \cdot \overline{u_{(2)}} .  $$

Define $\cv = \av + \bv$ and $\dv = \av\bv + \bv\av$.
Observe that $\deg(\cv) = 1$ and $\deg(\dv) = 2$.
In what follows we need to
consider a linear order on $\cd$-monomials
of degree $n$.
Let $u$ and $v$ be two $\cd$-monomials
$u = \cv^{i_{0}} \dv \cv^{i_{1}} \dv \cdots \dv \cv^{i_{p}}$
and
$v = \cv^{j_{0}} \dv \cv^{j_{1}} \dv \cdots \dv \cv^{j_{q}}$.
If $u$ contains fewer occurrences of the variable~$\dv$ than~$v$ 
(that is, $p < q$), then set $u < v$.
If $u$ contains the same number of occurrences of the variable~$\dv$ as $v$
($p = q$),
and
the vector
$(i_{0}, i_{1}, \ldots, i_{p}) <_{lex}
 (j_{0}, j_{1}, \ldots, j_{p})$
in lexicographic order $<_{lex}$,
then set $u < v$.

\begin{lemma}
Let $R$ be a ring and let $S$ be a subring of $R$.
Then the following intersection holds:
$$
    \Rcd \cap \Sab = \Scd .
$$
In other words,
when any $\cd$-polynomial $w$ with coefficients in $R$
is expanded as an $\ab$-polynomial and has
coefficients in the subring $S$,
then all the coefficients
of $w$, written as a $\cd$-polynomial, already belong to
the subring $S$.
\label{lemma_coefficients}
\end{lemma}
\begin{proof}
The containment
$\Rcd \cap \Sab \supseteq \Scd$
is clear.
It is enough to prove the
reverse containment for homogeneous $\cd$-polynomials
of degree $n$.
To derive a contradiction,
assume that there is a $\cd$-polynomial $w$
belonging to $\Rcd \cap \Sab$ but not to $\Scd$.
This means there is a $\cd$-monomial
in $w$ whose coefficient does not lie in the subring~$S$.
Let
$u = \cv^{i_{0}} \dv \cv^{i_{1}} \dv \cdots \dv \cv^{i_{p}}$
be the first such $\cd$-monomial in $w$
with respect to the previously described linear order.
Consider the $\ab$-monomial
$z = \av^{i_{0}} \bv \av \av^{i_{1}} \bv \av \cdots \bv \av \av^{i_{p}}$.
The $\ab$-monomial $z$ occurs when expanding $u$
into an $\ab$-polynomial. Observe that any other
$\cd$-monomial $v$ that has $z$ 
occurring in its $\ab$-expansion
must satisfy $v < u$ in the linear order.
Note that the coefficient of $z$ in the $\ab$-polynomial $w$
lies in the subring $S$. This coefficient is the
sum of certain coefficients of the $\cd$-polynomial $w$
where all but one (the coefficient of $u$)
belong to the subring $S$. This contradicts the assumption
that the coefficient of $u$ does not belong to the subring.
Hence the intersection holds.
\end{proof}

\section{The $\cd$-index}
\label{section_the_cd_index}

It is now natural to ask when the $\ab$-index
of a directed graph can be written
as a $\cd$-index,
that is,
when is it
an element of $\Zzz\langle \cv,\dv \rangle$ with
$\cv = \av + \bv$ and $\dv = \ab + \ba$.
To do this, we introduce the {\em rising}
and {\em falling polynomial} of an interval $[x,y]$
of a directed graph,
denoted
$  \widetilde{r}_{x,y}(q)$ and
$  \widetilde{f}_{x,y}(q)$.
After investigating their properties,
we relate these polynomials with
the aforementioned 
ability to express the non-homogeneous
$\ab$-index of as a
non-homogeneous
$\cd$-index.

A directed path $p = (e_{1}, e_{2}, \ldots, e_{k})$
in a labeled digraph $G$ is
called {\em rising} if
$\lambda(e_{i}) \sim \lambda(e_{i+1})$
for all $i = 1, \ldots, k-1$.
Similarly, a path $p$ is called {\em falling} if
$\lambda(e_{i}) \not\sim \lambda(e_{i+1})$
for all $i = 1, \ldots, k-1$.
For $x < y$
let $\widetilde{r}_{x,y}(q)$ be the polynomial
$$  \widetilde{r}_{x,y}(q)
   =
    \sum_{p \in {\mathcal R}(x,y)} q^{\ell(p) - 1}   ,  $$
where the sum ranges over all rising paths $p$ from $x$ to $y$.
Similarly, let $\widetilde{f}_{x,y}(q)$ be the polynomial
$$  \widetilde{f}_{x,y}(q)
   =
    \sum_{p \in {\mathcal F}(x,y)} q^{\ell(p) - 1}   ,  $$
where the sum ranges over all falling paths $p$ from $x$ to $y$.

Define two algebra maps $\kappa$ and $\lambda$
on $\zab$ by letting
$$   \begin{array}{l l l}
   \kappa(\av) = \av-\bv, &
   \kappa(\bv) = 0, &
   \kappa(1) = 1, \\
   \lambda(\av) = 0, &
   \lambda(\bv) = \bv-\av, &
   \lambda(1) = 1. 
     \end{array}     $$
The map $\kappa$ appeared first in
the paper~\cite[Section~10]{Ehrenborg_Readdy_c},
whereas the $\lambda$ map is more recent;
see~\cite[Section~2.2.2]{Slone}.
Observe these two maps are related by
$\overline{\kappa(u)} = \lambda(\overline{u})$.
The $\kappa$ and $\lambda$ maps allows one to
recapture the $\widetilde{r}$- and $\widetilde{f}$-polynomials
from the $\ab$-index $\Psi([x,y])$ as follows.
\begin{lemma}
For an interval $[x,y]$ in a labeled digraph $G$,
\begin{align}
\kappa(\Psi([x,y]))  & =  \widetilde{r}_{x,y}(\av-\bv) ,  \\
\lambda(\Psi([x,y])) & =  \widetilde{f}_{x,y}(\bv-\av) .
\end{align}
\end{lemma}
\begin{proof}
Since $\kappa(\bv) = 0$, the algebra map $\kappa$ applied
to an $\ab$-monomial only preserves the pure $\av$-terms,
and then replaces each $\av$ with $\av-\bv$.
Hence $\kappa(\Psi([x,y]))$ enumerates the rising chains.
A symmetric argument proves the second identity.
\end{proof}

\begin{lemma}
For any $\ab$-polynomial $u$ the following two identities
hold:
\begin{align}
u
& =
\kappa(u) + \sum_{u} \kappa(u_{(1)}) \cdot \bv \cdot u_{(2)} , 
\label{equation_kappa} \\
u
& =
\lambda(u) + \sum_{u} \lambda(u_{(1)}) \cdot \av \cdot u_{(2)} .
\label{equation_lambda}
\end{align}
\end{lemma}
\begin{proof}
Since equation~\eqref{equation_kappa} is
linear in $u$, it is enough to prove it for
$\ab$-monomials. We proceed by induction
on the degree of an $\ab$-monomial.
The three base cases $u = 1$, $u = \av$ and $u = \bv$
are straightforward to verify.
Assume now that equation~\eqref{equation_kappa} 
holds for the $\ab$-monomials $v$ and $w$. Then it also holds
for the product $v \cdot w$ by the following
calculation. The right-hand side of the identity~\eqref{equation_kappa}
in the case $u = v \cdot w$
is equal to
\begin{align*}
   \kappa(v \cdot w)
& +
 \sum_{w} \kappa(v \cdot w_{(1)}) \cdot \bv \cdot w_{(2)} 
 + \sum_{v} \kappa(v_{(1)}) \cdot \bv \cdot v_{(2)} \cdot w \\
& =
   \kappa(v) \cdot
\left(
   \kappa(w)
 + \sum_{w} \kappa(w_{(1)}) \cdot \bv \cdot w_{(2)}
\right)
 + \sum_{v} \kappa(v_{(1)}) \cdot \bv \cdot v_{(2)} \cdot w \\
& =
   \kappa(v) \cdot w
 + \sum_{v} \kappa(v_{(1)}) \cdot \bv \cdot v_{(2)} \cdot w \\
& =
   v \cdot w .
\end{align*}
The second identity~\eqref{equation_lambda}
follows by applying the involution 
$u \longmapsto \overline{u}$
and the relation
$\overline{\kappa(u)} = \lambda(\overline{u})$.
\end{proof}

\begin{remark}
{\rm
Equation~\eqref{equation_kappa}
is really Stanley's recursion~\cite{Stanley_d} for the
$\ab$-index of a graded poset
with rank function $\rho$, that is,
$$   \Psi([x,y])
   =
     (\av-\bv)^{\rho(x,y)-1}
   +
     \sum_{x < z < y}
            (\av-\bv)^{\rho(x,z)-1}
         \cdot
            \bv
         \cdot
            \Psi([z,y])  .  $$
This recursion follows directly by conditioning
on the first non-zero element in a chain.
Equation~\eqref{equation_kappa}
can be proven by using the fact that
$\kappa(\Psi([x,y])) = (\av-\bv)^{\rho(x,y)-1}$,
the $\ab$-index is
a coalgebra homomorphism~\cite{Ehrenborg_Readdy_c},
and the $\ab$-indexes of graded posets span $\zab$.
}
\label{remark_Stanley_recursion}
\end{remark}

\begin{remark}
{\rm 
The coalgebra $\zab$ does not have a counit.
Philosophically speaking, the two identities~\eqref{equation_kappa}
and~\eqref{equation_lambda} are a replacement for the defining
relation of the counit since they both allow
us to recapture the polynomial $u$ after applying the coproduct $\Delta$.
}
\end{remark}

Recall that the two non-commutative
variables $\cv$ and $\dv$
are defined by $\cv = \av + \bv$
and $\dv = \av \cdot \bv + \bv \cdot \av$,
with
$\cv$ of degree $1$ and $\dv$ of degree $2$.
The next lemma shows that $\ab$-polynomials
of a certain form are indeed $\cd$-polynomials.

\begin{lemma}
Let $p(x)$ and $q(x)$ be two polynomials in $\zx$
such that their odd degree terms agree, that is,
$p(x) - p(-x) = q(x) - q(-x)$. Then
\begin{align}
&     p(\av-\bv) + q(\bv-\av)
\in \zcd, 
\label{equation_Stanley_i} \\
&     p(\av-\bv) \cdot \bv + p(\bv-\av) \cdot \av  
\in \zcd.
\label{equation_Stanley_ii} 
\end{align}
\label{lemma_Stanley}
\end{lemma}
\begin{proof}
First note that
$(\av-\bv)^{2 \cdot k} = (\bv-\av)^{2 \cdot k} = 
(\cv^{2} - 2 \cdot \dv)^{k}$
and
$  (\av-\bv)^{2 \cdot k + 1}
 + (\bv-\av)^{2 \cdot k + 1}
 = 
   0$.
Hence by linearity it 
follows that the polynomial in~\eqref{equation_Stanley_i}
is a $\cd$-polynomial for all
polynomials $p(x)$ and~$q(x)$
satisfying the condition of the lemma.
The fact that the polynomial in~\eqref{equation_Stanley_ii}
is a $\cd$-polynomial follows again by 
linearity and by considering the parity of
the power of the monomial $x^{n}$.
For even powers we have
$  (\av-\bv)^{2 \cdot k} \cdot \bv
 + (\bv-\av)^{2 \cdot k} \cdot \av
 = 
   (\cv^{2} - 2 \cdot \dv)^{k} \cdot \cv$
and for odd powers we have
$  (\av-\bv)^{2 \cdot k + 1} \cdot \bv
 + (\bv-\av)^{2 \cdot k + 1} \cdot \av
 = 
   - (\cv^{2} - 2 \cdot \dv)^{k+1}$.
\end{proof}

\begin{remark}
{\rm
Equations~\eqref{equation_Stanley_i}
and~\eqref{equation_Stanley_ii}
in Lemma~\ref{lemma_Stanley}
can be viewed as linearizations
of statements due to
Stanley~\cite{Stanley_d}.
}
\label{remark_Stanley_extended}
\end{remark}

We now come to the main result of this section.

\begin{theorem}
For a labeled acyclic digraph $G$, the following
three statements are equivalent:
\begin{itemize}
\item[(i)]
For every interval $[x,y]$ in the digraph $G$
and for every non-negative integer $k$,
the number of rising paths from $x$ to $y$ of length $k$
is equal to
the number of falling paths from $x$ to $y$ of length~$k$.

\item[(ii)]
For every interval $[x,y]$ in the digraph $G$
and for every {\em even} positive integer $k$,
the number of rising paths from $x$ to $y$ of length $k$
is equal to
the number of falling paths from $x$ to $y$ of length~$k$.

\item[(iii)]
The $\ab$-index of every interval $[x,y]$ in the digraph $G$,
where $x < y$, is a polynomial in $\zcd$.
\end{itemize}
\label{theorem_cd}
\end{theorem}

\begin{definition}
A labeled acyclic digraph $G$ is said to be {\em balanced}
if it satisfies condition $(i)$ in
Theorem~\ref{theorem_cd}. Such a labeling is called
a {\em balanced labeling}.
\end{definition}
A different way to express the balanced condition is
that
$\widetilde{r}_{x,y}(q)  =  \widetilde{f}_{x,y}(q)$
for all pairs of elements~$x$ and~$y$ such that $x < y$.

\begin{figure}
\setlength{\unitlength}{1.3 mm}
\begin{center}
\begin{picture}(110,30)(0,0)
\put(0,5){
\begin{picture}(50,20)(-25,0)
\thicklines
\qbezier(0,0)(-12.5,5)(-25,10)
\qbezier(0,0)(-6,5)(-12,10)
\qbezier(0,0)(6,5)(12,10)
\qbezier(0,0)(12.5,5)(25,10)
\qbezier(0,20)(-12.5,15)(-25,10)
\qbezier(0,20)(-6,15)(-12,10)
\qbezier(0,20)(6,15)(12,10)
\qbezier(0,20)(12.5,15)(25,10)

\qbezier(0,0)(-8,10)(0,20)
\qbezier(0,0)(0,10)(0,20)
\qbezier(0,0)(8,10)(0,20)

\put(0,0){\circle*{1}}
\put(-25,10){\circle*{1}}
\put(-12,10){\circle*{1}}
\put(12,10){\circle*{1}}
\put(25,10){\circle*{1}}
\put(0,20){\circle*{1}}

\put(-18,4){\large $1$}
\put(-8.5,4){\large $1$}

\put(-18,14){\large $2$}
\put(-9,14){\large $2$}

\put(16,4){\large $2$}
\put(7.5,4){\large $2$}

\put(16,14){\large $1$}
\put(8,14){\large $1$}

\put(-6,9){\large $1$}
\put(-2,9){\large $2$}
\put(2,9){\large $3$}

\end{picture}
}

\put(70,0){
\begin{picture}(40,30)(0,0)
\thicklines
\qbezier(20,0)(10,5)(0,10)
\qbezier(20,0)(15,5)(10,10)
\qbezier(20,0)(20,5)(20,10)
\qbezier(20,0)(25,5)(30,10)
\qbezier(20,0)(30,5)(40,10)
\multiput(0,10)(10,0){5}{\qbezier(0,0)(0,5)(0,10)}
\multiput(0,10)(10,0){4}{\qbezier(0,0)(5,5)(10,10)}
\qbezier(40,10)(20,15)(0,20)
\qbezier(20,30)(10,25)(0,20)
\qbezier(20,30)(15,25)(10,20)
\qbezier(20,30)(20,25)(20,20)
\qbezier(20,30)(25,25)(30,20)
\qbezier(20,30)(30,25)(40,20)

\put(20,0){\circle*{1}}
\multiput(0,10)(10,0){5}{\circle*{1}}
\multiput(0,20)(10,0){5}{\circle*{1}}
\put(20,30){\circle*{1}}

\put(4,5){\large $2$}
\put(11,5){\large $2$}
\put(18,5){\large $2$}
\put(23,5){\large $2$}
\put(28,5){\large $2$}

\put(-2,14){\large $1$}
\put(8,14){\large $1$}
\put(18,12){\large $1$}
\put(28,14){\large $1$}
\put(38,14){\large $1$}

\put(2,14){\large $3$}
\put(12,14){\large $3$}
\put(23,15){\large $3$}
\put(33,15){\large $3$}
\put(34,9){\large $3$}

\put(4,23){\large $2$}
\put(11,23){\large $2$}
\put(18,23){\large $2$}
\put(23,23){\large $2$}
\put(28,23){\large $2$}

\end{picture}
}

\end{picture}
\end{center}

\caption{Two balanced directed graphs
where the relation on the labeled set $\Lambda = \{1,2,3\}$ is
the natural linear order.
Their respective $\cd$-indexes are
$2 \cdot \cv + 3$
and
$5 \cdot \dv$.
These two examples show that the
$\cd$-index of a graph is not necessarily homogeneous
and that the coefficient of the $\cv$-power term
is not necessarily $1$.
}
\label{figure_one}
\end{figure}

\begin{example}
{\rm
See Figure~\ref{figure_one} for two examples
of balanced digraphs and their corresponding $\cd$-indexes. 
In Figure~\ref{figure_two} we give a
labeled digraph with two different
relations on the underlying label set.
Each yields a balanced digraph.
Note the resulting $\cd$-indexes differ, with
the second relation yielding a negative coefficient. 
}
\end{example}

\begin{figure}
\setlength{\unitlength}{1 mm}
\begin{center}
\begin{picture}(100,30)(0,0)
\put(0,0){
\begin{picture}(10,30)(0,0)
\thicklines
\qbezier(5,0)(5,5)(5,10)
\qbezier(5,10)(0,15)(5,20)
\qbezier(5,10)(10,15)(5,20)
\qbezier(5,20)(5,25)(5,30)
\multiput(5,0)(0,10){4}{\circle*{1}}
\put(1,4){\large $\beta$}
\put(-2,14){\large $\alpha$}
\put(10,14){\large $\gamma$}
\put(1,23){\large $\beta$}

\end{picture}
}

\put(50,0){
\begin{picture}(50,30)(0,0)
\put(0,25){\large Relation (i):}
\put(25,25){\large $\alpha \sim \beta \sim \gamma$}
\put(25,20){\large $\gamma \not\sim \beta \not\sim \alpha$}
\put(0,10){\large Relation (ii):}
\put(25,10){\large $\alpha \sim \beta \sim \alpha$}
\put(25,5){\large $\gamma \not\sim \beta \not\sim \gamma$}
\end{picture}
}

\end{picture}
\end{center}

\caption{A labeled directed graph
and two different relations on the label set
$\Lambda = \{\alpha, \beta, \gamma\}$.
Both relations yield a balanced graph.
Relation (i) is the linear order
and  the $\cd$-index is
$\av\bv + \bv\av = \dv$,
whereas relation (ii)
gives the $\cd$-index 
$\av\av + \bv\bv = \cv^{2} - \dv$.}
\label{figure_two}
\end{figure}

\begin{remark}
{\rm
Condition $(ii)$, that is, it suffices to check
the balanced condition for
paths of even length, has a corresponding statement
for graded posets.
A graded poset $P$ of odd rank
satisfying that every proper interval
of $P$ is Eulerian is also an Eulerian poset.
See~\cite[Chapter~3, Exercise~69(c)]{Stanley_EC_I}.
}
\end{remark}

\begin{remark}
{\rm
Consider graded posets with an $R$-labeling.
In this case,
the balanced condition 
implies that the number of rising chains (namely $1$)
in an interval $[x,y]$ of rank $k+1$
is equal to the number of falling chains,
that is, 
$1 = h_{\emptyset}([x,y]) = h_{\{1, \ldots, k\}}([x,y])$.
By Hall's theorem on the M\"obius function, this can
be stated as
$\mu(x,y) = (-1)^{\rho(x,y)}$. Since this relation holds for
all intervals $[x,y]$, this implies that the poset is Eulerian
and hence the $\cd$-index exists.  This result is classical;
see~\cite{Bayer_Klapper, Stanley_d}.
This is reminiscent of the work in~\cite{Billera_Liu},
where it was shown that if the Euler--Poincar\'e relation
holds for every interval in a poset  then the poset
satisfies the generalized Dehn--Sommerville relations and
has a $\cd$-index.
}
\end{remark}

\begin{proof}[Proof of Theorem~\ref{theorem_cd}.]
The implication that $(i) \Longrightarrow (ii)$ is clear.
For $(iii) \Longrightarrow (i)$,
observe that the variables $\cv$ and $\dv$
are symmetric in $\av$ and $\bv$. Hence
$\widetilde{r}_{x,y}(q)  =  \Psi([x,y])|_{\av = q, \bv = 0}
                     =  \Psi([x,y])|_{\cv = q, \dv = 0}
                     =  \Psi([x,y])|_{\av = 0, \bv = q}
                     =  \widetilde{f}_{x,y}(q)$.

Finally, assume that $(ii)$ is true and we will prove
the existence of the $\cd$-index $(iii)$.
The proof is by induction on the longest path in the interval $[x,y]$.
The base case is when the length of the longest path is $1$.
In this case the $\cd$-index is just the number of
edges between $x$ and $y$.
Assume now that the $\cd$-index exists for all subintervals
in $[x,y]$.
Add equations~\eqref{equation_kappa}
and~\eqref{equation_lambda} to obtain
$$
2 \cdot u 
   =
\kappa(u)
   +
\lambda(u)
   +
\sum_{u} \left(\kappa(u_{(1)})  \cdot \bv
            +
               \lambda(u_{(1)}) \cdot \av \right) \cdot u_{(2)} . $$
Now apply this equation to $u = \Psi([x,y])$,
the $\ab$-index of the entire interval $[x,y]$.
Since the $\ab$-index is a coalgebra homomorphism,
we have that
$$
2 \cdot \Psi([x,y])
  =
\widetilde{r}_{x,y}(\av-\bv)
   +
\widetilde{f}_{x,y}(\bv-\av)
   +
\sum_{x < z < y} 
         \left(\widetilde{r}_{x,z}(\av-\bv) \cdot \bv
            +
               \widetilde{f}_{x,z}(\bv-\av) \cdot \av\right)
                                               \cdot \Psi([z,y])
 . $$
By the induction hypothesis we know that 
$\Psi([x,z])$ and $\Psi([z,y])$
are both $\cd$-polynomials with integer coefficients.
By the implication $(iii) \Longrightarrow (i)$, we have
that $\widetilde{r}_{x,z}(q) = \widetilde{f}_{x,z}(q)$.
Thus by 
equation~\eqref{equation_Stanley_ii}
of Lemma~\ref{lemma_Stanley}, 
we know
that the term inside the summation sign is a $\cd$-polynomial
with integer coefficients.
Similarly, by 
equation~\eqref{equation_Stanley_i} of
the same lemma,
the sum of the two terms outside the 
summation sign is a $\cd$-polynomial
with integer coefficients.
Hence the expression $2 \cdot \Psi([x,y])$ is
a $\cd$-polynomial with integer coefficients.
By Lemma~\ref{lemma_coefficients}
with $R = \Qqq$ and $S = \Zzz$,
we conclude that the $\cd$-polynomial~$\Psi([x,y])$ has
integer coefficients.
\end{proof}

\begin{remark}
{\rm
The existence of the $\cd$-index for Eulerian posets
can be proved in a similar manner. Observe that
for any interval $[x,y]$ in a graded poset
we have that
$\lambda(\Psi([x,y]))
  =
 (-1)^{\rho(x,y)} \cdot \mu(x,y) \cdot (\bv-\av)^{\rho(x,y)-1}$.
Hence for an interval $[x,y]$
which satisfies the Eulerian condition we have that
\begin{align*}
\kappa(\Psi([x,y]))
& =
 (\av-\bv)^{\rho(x,y)-1}  ,  \\
\lambda(\Psi([x,y]))   
& =
 (\bv-\av)^{\rho(x,y)-1}  .
\end{align*}
Now proceed as in the proof of Theorem~\ref{theorem_cd}.
}
\label{remark_new_proof}
\end{remark}

A different way to prove the implication 
$(ii) \Longrightarrow (iii)$ 
in Theorem~\ref{theorem_cd} is to use the homology
techniques developed in~\cite{Ehrenborg_Readdy_homology}.
Let $R$ be a commutative ring with a unit and let $A$
be an $R$-module with a coassociative coproduct $\Delta$.
Let $d_{n}: A^{\tensor n} \longrightarrow A^{\tensor (n+1)}$
denote the map
$$    d_{n} 
    = 
      \sum_{i+j=n-1}
             (-1)^{i}
          \cdot
             \id^{\tensor i}
          \cdot
             \Delta
          \cdot
             \id^{\tensor j}   .   $$
The coassociativity of the coproduct $\Delta$
implies that $d_{n} \circ d_{n+1} = 0$, 
that is, $d_{n}$ is the boundary map of a chain complex.
In~\cite{Ehrenborg_Readdy_homology}
the Hochschild cohomology is computed for the chain complex
\begin{equation}
       0
  \longrightarrow
       A
  \stackrel{d_{1}}{\longrightarrow}
       A^{\tensor 2}
  \stackrel{d_{2}}{\longrightarrow}
       A^{\tensor 3}
  \longrightarrow
       \cdots
\label{equation_chain_complex}
\end{equation}
when $A$ is the Newtonian coalgebra $\Rcd$.
Theorem~4.1 in~\cite{Ehrenborg_Readdy_homology}
states when the ring $R$ has $2$ as a unit,
the cohomology vanishes everywhere except in
the bottom cohomology.
Armed with this result, we can give a different proof.

\begin{proof}[Second proof of the implication
$(ii) \Longrightarrow (iii)$ in Theorem~\ref{theorem_cd}.]
Let $R$ be a ring that contains the integers and has~$2$ as a unit.
(One such example is $R = \Qqq$.)
The proof is again by induction on the longest path
in the interval $[x,y]$.
Since the $\ab$-index is a coalgebra homomorphism and by
the induction hypothesis, we have that
$$   \Delta(\Psi([x,y]))
   =
     \sum_{x < z < y}
            \Psi([x,z]) \tensor \Psi([z,y])   
   \in
     \Rcd \tensor \Rcd    . $$
Since $\Delta$ is coassociative, we also have
$d_{2}(\Delta(\Psi([x,y])) = 0$, that is,
the element $\Delta(\Psi([x,y])) = d_{1}(\Psi([x,y]))$
lies in the kernel of the map $d_{2}$.
The chain complex~\eqref{equation_chain_complex}
is exact at this point, so there
is an element $w \in \Rcd$ such that
$\Delta(w) = d_{1}(w) = \Delta(\Psi([x,y]))$.
Hence $w$ and $\Psi([x,y]))$ differ by
an element in the kernel
of $\Delta : \Rab \longrightarrow \Rab \tensor \Rab$.
The kernel of $\Delta$ is
$R[\av-\bv]$, so we have
that
$\Psi([x,y]) - w = p(\av-\bv)$ for some polynomial $p(x)$.

Let $n$ be an odd positive integer.
Condition~$(ii)$ states that
the number of rising paths from $x$ to~$y$ of length $n+1$
is equal to
the number of falling paths from $x$ to $y$ of length $n+1$.
This is equivalent to the condition that  
the coefficients of $\av^{n}$ and $\bv^{n}$
in $\Psi([x,y])$ are identical.
Since $w$ is a $\cd$-polynomial,
the coefficients of $\av^{n}$ and $\bv^{n}$
are also the same.
Hence 
the coefficients of $\av^{n}$ and $\bv^{n}$
in $p(\av-\bv)$ are the same, proving that
the polynomial $p$ only has even degree terms,
that is, $p(\av-\bv)$ is a polynomial in the variable
$(\av-\bv)^{2} = \cv^{2} - 2 \cdot \dv$.
Hence $\Psi([x,y])$ belongs
to $\Rcd$. Again by
Lemma~\ref{lemma_coefficients}
we have $\Psi([x,y]) \in \zcd$.
\end{proof}

\section{Alexander duality}
\label{section_Alexander}

In this section we assume all the digraphs under
consideration are bounded with minimal element~$\hz$
and maximal element~$\ho$.
For a labeled acyclic digraph~$G$ with vertex set~$V$,
we define the restricted digraph~$G_{S}$ where $S$ is
a subset of $V - \{\hz,\ho\}$. 
The edge label set is given by
$$  \Lambda^{+}   =   \bigcup_{n \geq 1} \Lambda^{n}   ,  $$
and the relation $\sim$ on $\Lambda^{+}$ is
$(\lambda_{1}, \ldots, \lambda_{m}) \sim (\mu_{1}, \ldots, \mu_{n})$
if and only if
$\lambda_{m} \sim \mu_{1}$.
The vertex set
of~$G_{S}$ is $S \cup \{\hz,\ho\}$.
For every rising path
$p = (e_{1}, \ldots, e_{k})$ in the digraph $G$
which starts and ends in $S \cup \{\hz,\ho\}$
but none of the intermediate vertices are in $S$,
that is,
$\tail(e_{1}), \head(e_{k}) \in S \cup \{\hz,\ho\}$
but
$\head(e_{2}), \ldots, \head(e_{k-1}) \not\in S$,
let there be a directed edge in $G_{S}$ from $\tail(e_{1})$
to $\head(e_{k})$
with the label $(\lambda(e_{1}), \ldots, \lambda(e_{k}))$
in $\Lambda^{+}$.

For two vertices $x$ and $y$ in the restricted graph $G_{S}$
observe that the number of rising paths from~$x$ to~$y$
is the same as the number of rising paths in the graph $G$.
This follows since a path in $G_{S}$ corresponds
to a path in $G$ as follows.
Let $p^{\prime} = (e^{\prime}_{1}, \ldots, e^{\prime}_{j})$
be a path in $G_{S}$. 
We obtain a path~$p$ in~$G$ by 
concatenating the rising paths that are associated
with the edges $e^{\prime}_{i}$.
Furthermore the condition that the path $p^{\prime}$
is rising in $G_{S}$
is exactly that the path $p$ is rising in $G$,
since the only condition that needs to be verified
is that $p$ is rising at the gluing vertices
$\head(e^{\prime}_{1}), \ldots, \head(e^{\prime}_{j-1})$.

Let $\ell(G)$ denote the length of the longest
path in the digraph $G$. 
We say that an acyclic digraph has the {\em parity condition}
if the length of every path from the source~$\hz$ to the sink
$\ho$ has the same parity. 
Then in a digraph which has
the parity condition, the length of any path 
from $\hz$ to $\ho$ is congruent to~$\ell(G)$
modulo~$2$.

\begin{figure}
\setlength{\unitlength}{1.5 mm}
\begin{center}
\begin{picture}(20,30)(0,0)
\thicklines
\qbezier(10,0)(10,0)(0,10)
\qbezier(10,0)(10,0)(10,10)
\qbezier(10,0)(10,0)(20,10)
\qbezier(0,10)(0,10)(0,20)
\qbezier(0,10)(0,10)(10,20)
\qbezier(10,10)(10,10)(0,20)
\qbezier(10,10)(10,10)(20,20)
\qbezier(20,10)(20,10)(10,20)
\qbezier(20,10)(20,10)(20,20)
\qbezier(0,20)(0,20)(10,30)
\qbezier(10,20)(10,20)(10,30)
\qbezier(20,20)(20,20)(10,30)

\put(10,0){\circle*{1}}
\multiput(0,10)(10,0){3}{\circle*{1}}
\multiput(0,20)(10,0){3}{\circle*{1}}
\put(10,30){\circle*{1}}

\put(2,4){\large $1$}
\put(8,4){\large $2$}
\put(16,4){\large $3$}

\put(-2,14){\large $2$}
\put(3,11){\large $3$}
\put(6,11){\large $1$}
\put(13,11){\large $3$}
\put(16,11){\large $1$}
\put(21,14){\large $2$}

\put(2,24){\large $3$}
\put(8,24){\large $2$}
\put(17,24){\large $1$}

\end{picture}
\hspace*{30 mm}
\begin{picture}(20,30)(0,0)
\thicklines
\qbezier(10,0)(10,0)(0,10)
\qbezier(0,10)(0,10)(10,20)
\qbezier(10,20)(10,20)(10,30)

\qbezier(0,10)(-5,23)(10,30)

\put(10,0){\circle*{1}}
\put(0,10){\circle*{1}}
\put(10,20){\circle*{1}}
\put(10,30){\circle*{1}}

\put(3,3){\large $1$}
\put(5,13){\large $3$}
\put(8,24){\large $2$}

\put(-2,23){\large $23$}

\end{picture}
\hspace*{20 mm}
\begin{picture}(20,30)(0,0)
\thicklines
\qbezier(10,0)(10,0)(10,10)
\qbezier(10,0)(10,0)(20,10)
\qbezier(10,10)(10,10)(0,20)
\qbezier(10,10)(10,10)(20,20)
\qbezier(20,10)(20,10)(20,20)
\qbezier(0,20)(0,20)(10,30)
\qbezier(20,20)(20,20)(10,30)

\qbezier(0,20)(-5,7)(10,0)
\qbezier(20,10)(35,23)(10,30)

\put(10,0){\circle*{1}}
\put(10,10){\circle*{1}}
\put(20,10){\circle*{1}}
\put(0,20){\circle*{1}}
\put(20,20){\circle*{1}}
\put(10,30){\circle*{1}}

\put(8,4){\large $2$}
\put(16,4){\large $3$}
\put(4,13){\large $1$}
\put(15,13){\large $3$}
\put(21,14){\large $2$}
\put(3,25){\large $3$}
\put(17,24){\large $1$}

\put(-2,5){\large $12$}
\put(23,25){\large $12$}

\end{picture}
\end{center}

\caption{The Boolean algebra $G = B_{3}$
with its classical $R$-labeling
$\lambda(I \rightarrow I \cupdot \{i\}) = i$,
and the two restricted digraphs
$G_{S}$ and $G_{T}$, where
$S = \{\{1\}, \{1,3\}\}$
and
$T = \{\{2\}, \{3\}, \{1,2\}, \{2,3\}\}$.
Observe that $G_{S}$ has no falling paths,
whereas $G_{T}$ has two falling paths
of lengths $2$ and $3$.
We have
$\tilde{f}_{G_S}(q) = 0$
and 
$\tilde{f}_{G_T}(q) = q + q^2$,
implying
$\tilde{f}_{G_S}(-1) = 0$
and
$\tilde{f}_{G_R}(-1) = 0$.
}
\label{figure_three}
\end{figure}

We can now formulate Alexander duality for balanced digraphs.
See Figure~\ref{figure_three} for an illustration of this theorem.
\begin{theorem}
[Alexander duality for balanced digraphs]
Let $G$ be a balanced acyclic digraph that satisfies the parity
condition. Let the vertex set have the partition
$V = S \cupdot T \cupdot \{\hz,\ho\}$. Then the falling paths
in the two restricted digraph $G_{S}$ and~$G_{T}$
satisfy the identity
$$
    \widetilde{f}_{G_{S}}(-1)
  =
    (-1)^{\ell(G) - 1}
  \cdot
    \widetilde{f}_{G_{T}}(-1)  .    $$
\label{theorem_Alexander}
\end{theorem}

Before proving this theorem, we must establish one more result.
For a path $p = (e_{1}, \ldots, e_{k})$
in the digraph $G$, 
recall that $i(p)$ is
the set of all interior vertices of the path,
that is,
$i(p) = \{\head(e_{1}), \ldots, \head(e_{k-1})\}$.
Note that $|i(p)| = \ell(p) -1$.
Furthermore, let
$\Asc(p)$ and $\Des(p)$ denote the set of vertices where
the path $p$ has ascents, respectively, descents, that is,
\begin{align*}
\Asc(p)
& =
\{ \head(e_{i}) \: : \: \lambda(e_{i}) \sim \lambda(e_{i+1})\} , \\
\Des(p)
& =
\{ \head(e_{i}) \: : \: \lambda(e_{i}) \not\sim \lambda(e_{i+1})\} .
\end{align*}
Directly
$i(p)$ is the disjoint union of $\Asc(p)$ and $\Des(p)$.

\begin{proposition}
Let $G$ be a bounded labeled acyclic digraph 
such that in every interval the number of rising paths equals
the number of falling paths.
Let the vertex set of $G$ have the partition
$V = S \cupdot T \cupdot \{\hz,\ho\}$. 
Then the following two sums are equal:
$$
\sum_{\substack{p \\ \Asc(p) \subseteq T \\ \Des(p) \subseteq S}}
(-1)^{|i(p) \cap S|} 
  =
\sum_{\substack{p \\ \Asc(p) \subseteq S \\ \Des(p) \subseteq T}}
(-1)^{|i(p) \cap S|} .
$$
\label{proposition_Alexander}
\end{proposition}
\begin{proof}
Let $\A(S)$ and $\B(S)$ denoted the left-hand side of
the identity, respectively, the right-hand side.
The proof is by double induction.
First we induct over the longest path in the digraph.
Here the induction base is $\ell(G) = 1$,
that is, the graph consists only of the source and the sink.
Each path has length $1$ and is both rising and falling.
Thus the statement is immediate.

Now assume that the statement holds for all digraphs of length
less than $\ell(G)$. 
We induct on the set~$S$.
The induction basis is when $S$ is empty.
Then $\A(\emptyset)$ and $\B(\emptyset)$
are the number of rising, respectively, falling chains
in the graph $G$. The balanced condition
implies that they are equal, completing the induction basis.

For the induction step, assume that $\A(S) = \B(S)$ for a set $S$.
We will prove it for $S \cup \{x\}$ 
where $x$ is an element not in $S$.
Observe that
\begin{align*}
\A(S \cup \{x\}) - \A(S)
& =
\sum_{\substack{p \\ \Asc(p) \subseteq T - \{x\} \\ x \in \Des(p) \subseteq S \cup \{x\}}}
(-1)^{|i(p) \cap (S \cup \{x\})|}
-
\sum_{\substack{p \\ x \in \Asc(p) \subseteq T \\ \Des(p) \subseteq S}}
(-1)^{|i(p) \cap S|} \\
& =
-
\sum_{\substack{p \\ \Asc(p) \subseteq T - \{x\} \\ x \in \Des(p) \subseteq S \cup \{x\}}}
(-1)^{|i(p) \cap S|}
-
\sum_{\substack{p \\ x \in \Asc(p) \subseteq T \\ \Des(p) \subseteq S}}
(-1)^{|i(p) \cap S|} .
\end{align*}
Combining these two sums we obtain a sum over all paths $p$
through the vertex $x$ such that
$\Asc(p) - \{x\} \subseteq T$
and
$\Des(p) - \{x\} \subseteq S$.
That is, there is no condition on the path at the vertex $x$.
Hence any such path is the concatenation of a path $p_{1}$
in $[\hz, x]$ and a path $p_{2}$ in $[x, \ho]$.
Using that
$|i(p) \cap S| = |i(p_{1}) \cap S| + |i(p_{2}) \cap S|$,
the difference $\A(S \cup \{x\}) - \A(S)$ is given by
\begin{align*}
\A(S \cup \{x\}) - \A(S)
& =
 -
\sum_{\substack{p_{1} \\ \Asc(p_{1}) \subseteq T \\ \Des(p_{1}) \subseteq S}}
(-1)^{|i(p_{1}) \cap S|} 
\cdot
\sum_{\substack{p_{2} \\ \Asc(p_{2}) \subseteq T \\ \Des(p_{2}) \subseteq S}}
(-1)^{|i(p_{2}) \cap S|}
\\
& =
- \A_{[\hz,x]}(S \cap (\hz,x)) \cdot \A_{[x,\ho]}(S \cap (x,\ho)) ,
\end{align*}
where the first sum is over paths $p_{1}$ in $[\hz,x]$
and the second sum is over paths $p_{2}$ in $[x,\ho]$.
By applying the first induction hypothesis to the smaller
digraphs $[\hz,x]$ and $[x,\ho]$ we have
\begin{align*}
\A(S \cup \{x\}) - \A(S)
& =
- \B_{[\hz,x]}(S \cap (\hz,x)) \cdot \B_{[x,\ho]}(S \cap (x,\ho)) \\
& =
 -
\sum_{\substack{p_{1} \\ a(p_{1}) \subseteq S \\ d(p_{1}) \subseteq T}}
(-1)^{|i(p_{1}) \cap S|} 
\cdot
\sum_{\substack{p_{2} \\ a(p_{2}) \subseteq S \\ d(p_{2}) \subseteq T}}
(-1)^{|i(p_{2}) \cap S|}
\\
& =
-
\sum_{\substack{p \\ x \in i(p) \\ \Asc(p) - \{x\} \subseteq S \\ \Des(p) - \{x\} \subseteq T}}
(-1)^{|i(p) \cap S|} \\
& =
\sum_{\substack{p \\ x \in \Asc(p) \subseteq S \cup \{x\} \\ \Des(p) \subseteq T - \{x\}}}
(-1)^{|i(p) \cap (S \cup \{x\})|} 
-
\sum_{\substack{p \\ \Asc(p) \subseteq S \\ x \in \Des(p) \subseteq T}}
(-1)^{|i(p) \cap S|} \\
& =
\B(S \cup \{x\}) - \B(S) .
\end{align*}
Hence
$\A(S \cup \{x\}) = \B(S \cup \{x\})$ completing the induction.
\end{proof}

The statement of Proposition~\ref{proposition_Alexander}
is not symmetric in $S$ as the following
corollary illustrates.  
Also note the assumptions 
in Proposition~\ref{proposition_Alexander}
are not as strict as the balanced condition.

\begin{corollary}
Let $G$ be a labeled acyclic digraph 
such that in every interval the number of rising paths equals
the number of falling paths.
Then following two alternating sums agree:
$$
\sum_{\text{$p$ rising}}
(-1)^{\ell(p)}
  =
\sum_{\text{$p$ falling}}
(-1)^{\ell(p)}.
$$
\label{corollary_Alexander}
\end{corollary}
\begin{proof}
Take $T = \emptyset$ in Proposition~\ref{proposition_Alexander}.
\end{proof}

\begin{proof}[Proof of Theorem~\ref{theorem_Alexander}.]
Expanding $\widetilde{f}_{G_{S}}(-1)$ we have
\begin{align*}
\widetilde{f}_{G_{S}}(-1)
& =
\sum_{p^{\prime}} (-1)^{\ell(p^{\prime})-1}  ,
\end{align*}
where the sum is over all falling paths $p^{\prime}$ in $G_{S}$.
By replacing each edge in the path $p^{\prime}$ with
the associated rising path in $G$, we obtain a path $p$ in
the digraph $G$ such that
$\Asc(p) \subseteq T$
and
$\Des(p) \subseteq S$.
Hence
$$
\widetilde{f}_{G_{S}}(-1)
  =
\sum_{\substack{p \\ \Asc(p) \subseteq T \\ \Des(p) \subseteq S}}
(-1)^{|i(p) \cap S|} .
$$
By a symmetric argument we have
$$
   (-1)^{\ell(G) - 1} \cdot
   \widetilde{f}_{G_{T}}(-1) 
  =
   (-1)^{\ell(G) - 1} \cdot
\sum_{\substack{p \\ \Asc(p) \subseteq S \\ \Des(p) \subseteq T}}
(-1)^{|i(p) \cap T|} 
  =
\sum_{\substack{p \\ \Asc(p) \subseteq S \\ \Des(p) \subseteq T}}
(-1)^{|i(p) \cap S|} ,
$$
where we used the parity condition
that $|i(p) \cap S| + |i(p) \cap T| = i(p) \equiv \ell(G) - 1 \bmod 2$.
The duality result now follows from Proposition~\ref{proposition_Alexander}.
\end{proof}

\section{Application to Bruhat graphs}
\label{section_Bruhat}

An important application of balanced labeled graphs
is to the family of Bruhat graphs.
In this section we give a brief overview of Bruhat graphs.
For a more complete description of Coxeter systems,
we refer the reader to
the book of 
Bj\"orner and Brenti~\cite{Bjorner_Brenti}.

Let $(W,S)$ be a Coxeter system,
where $W$ denotes a (finite or infinite) Coxeter group
with generators~$S$ and   
$\ell(u)$ denotes the length of
a group element $u$.
Let $T$ be the set of reflections, that is,
$T = \{w \cdot s \cdot w^{-1} \:\: : \:\:
         s \in S, w \in W \}$.
The Bruhat graph has the group~$W$ as its vertex set
and its set of labels $\Lambda$ is the set of reflections~$T$.
The edges and their labeling are defined as follows.
There is a directed edge from $u$ to $v$ labeled $t$
if $u \cdot t = v$ and $\ell(u) < \ell(v)$.
The underlying poset of the Bruhat graph is called
the (strong) Bruhat order. It is important to note that every
interval of the Bruhat order is Eulerian,
that is,
every interval $[x,y]$ has M\"obius function
given by
$\mu(x,y) = (-1)^{\rho(y) - \rho(x)}$,
where $\rho$ denotes the rank function.

The motivation for studying the $\cd$-index of Bruhat
graphs is that the $\cd$-index of the interval $[u,v]$
determines the Kazhdan--Lusztig polynomial 
$P_{u,v}(q)$.  See~\cite[Section~3]{Billera_Brenti}.
The first step is to define the $R$-polynomials $R_{u,v}(q)$.
See~\cite[Theorem~5.1.1]{Bjorner_Brenti} for further details.
\begin{theorem}
There is a unique family of polynomials
$\{R_{u,v}(q)\}_{u,v \in W}$ with integer coefficients
satisfying the following conditions:
\begin{itemize}
\item[(i)]  $R_{u,v} = 0$ if $u \not\leq v$,
\item[(ii)] $R_{u,v} = 1$ if $u = v$, and
\item[(iii)] If $s \in S$ and $\ell(v \cdot s) < \ell(v)$
then
$$    R_{u,v}(q)
   =
        \begin{cases}
           R_{us,vs}(q) 
           & \text{ if } \ell(u \cdot s) < \ell(u), \\
           q \cdot R_{us,vs}(q) + (q-1) \cdot R_{u,vs}(q) 
           & \text{ if } \ell(u \cdot s) > \ell(u).
        \end{cases}
$$
\end{itemize}
\end{theorem}
A combinatorial interpretation of the $R$-polynomials
is given by Dyer~\cite{Dyer}.
See also~\cite[Proposition~5.3.1 and Theorem~5.3.4]{Bjorner_Brenti}.
On the set of reflections there exist 
conditions for a total ordering.
An ordering satisfying these conditions is called
a reflection ordering. The fact that a reflection ordering
exists follows
from~\cite[Proposition~5.2.1]{Bjorner_Brenti}.
Dyer's interpretation is
$$     R_{u,v}(q)
   =
       q^{\ell(u,v)/2}
     \cdot
       \widetilde{R}_{u,v}
           \left(
               q^{1/2} - q^{-1/2}
           \right)  ,   $$
where the
$\widetilde{R}$-polynomials
are defined in equation~\eqref{equation_R_tilde_F_tilde}
with respect to a reflection ordering of the set of reflections $T$.

We can now state and give a concise
proof of the first main result from~\cite{Billera_Brenti},
namely the existence of the complete $\cd$-index of the
Bruhat order. We prefer to call it {\em the $\cd$-index
of the Bruhat graph} to distinguish it from the
$\cd$-index of the Bruhat order.

\begin{theorem}[Billera--Brenti]
For an interval $[u,v]$ in the Bruhat order, where $u < v$,
the following three conditions hold:
\begin{itemize}
\item[(i)]
The interval $[u,v]$ in the Bruhat graph has
a $\cd$-index $\Psi([u,v])$.

\item[(ii)]
Restricting the $\cd$-index $\Psi([u,v])$
to those terms 
of degree $\ell(v) - \ell(u) - 1$ equals the
$\cd$-index of the graded poset $[u,v]$.

\item[(iii)]
The degree of a term in the $\cd$-index $\Psi([u,v])$
is less than or equal to $\ell(v) - \ell(u) - 1$
and has the same parity as $\ell(v) - \ell(u) - 1$.
\end{itemize}
\label{theorem_Billera_Brenti}
\end{theorem}
\begin{proof}
The reverse of a reflection ordering is also reflection ordering.
Hence the number of rising chains of length $k$
is equal to the number of falling chains of the same length.
Thus part $(i)$ follows from Theorem~\ref{theorem_cd}.
Part (ii) follows from the fact that 
when one restricts
the labeling to the poset structure of one interval $[u,v]$,
that is, only
considering the cover relations, the reflection ordering
is an $R$-labeling.
Part (iii) follows from the fact that the Bruhat graph
is bipartite.
\end{proof}

\section{Quasisymmetric functions}
\label{section_quasisymmetric_functions}

In this section
we review the basic set-up surrounding the 
ring of quasisymmetric functions.
For a bounded labeled digraph we introduce the  {\em rising} and
{\em falling quasisymmetric functions} and relate these
with a shift of the aforementioned rising and falling
polynomials.
We show the rising
and falling quasisymmetric functions  are Hopf algebra homomorphisms
from the Hopf algebra formed by the linear span of bounded
labeled acyclic digraphs
to the quasisymmetric functions.
We then reformulate Theorem~\ref{theorem_cd} in terms of
the peak algebra.

The connection between flag $f$-vectors of graded posets
and quasisymmetric functions was developed
by Ehrenborg~\cite{Ehrenborg_Hopf}.
The companion theory for edge labeled posets
and quasisymmetric functions
is due to Bergeron and Sottile~\cite{Bergeron_Sottile}.
The peak algebra was introduced by Stembridge~\cite{Stembridge}.
The link between the peak algebra
and the quasisymmetric functions of Eulerian posets was made
by Bergeron, Mykytiuk,
Sottile and van Willigenburg
in~\cite{Bergeron_Mykytiuk_Sottile_van_Willigenburg}.

Let $\Sigma_{n}$ denote the set of all
compositions of $n$, that is, all sequences 
$\alpha = (\alpha_{1},\alpha_{2},\ldots,\alpha_{m})$
of positive integers such that
$\alpha_{1} + \alpha_{2} + \cdots + \alpha_{m} = n$.
We form $\Sigma_{n}$ into a poset by
defining the cover relation
$(\alpha_{1},\ldots,\alpha_{i} + \alpha_{i+1},\ldots,\alpha_{m})
   \prec
 (\alpha_{1},\ldots,\alpha_{i}, \alpha_{i+1},\ldots,\alpha_{m})$.
Observe that the minimal element is the composition $(n)$
and the maximal element is the composition $(1,1, \ldots, 1)$.
In fact, for $n \geq 1$, the poset~$\Sigma_{n}$ is isomorphic
to the Boolean algebra~$B_{n-1}$. Note also that $\Sigma_{0}$
consists of the unique composition of the integer $0$.
Especially, note that each composition $\alpha$
in the poset $\Sigma_{n}$
has a unique complement that we denote by $\alpha^{c}$.
To find the complement, write the composition using
commas, plus signs and $1$'s, and exchange the commas and plus signs.
As an example, the complement of
$(3,1,2) = (1+1+1,1,1+1)$ is $(1,1,1+1+1,1) = (1,1,3,1)$.
Finally, let $\Sigma = \cup_{n \geq 0} \Sigma_{n}$.

A function $f$ in the ring
$\Zzz[[w_{1},w_{2},\ldots]]$ of
power series with bounded degree
is called {\em quasisymmetric} if for any sequence
of positive integers $\alpha_{1},\alpha_{2},\ldots,\alpha_{m}$ we have
$$
\left[w_{i_1}^{\alpha_{1}}\cdots w_{i_k}^{\alpha_{m}} \right] f
    =
\left[w_{j_1}^{\alpha_{1}}\cdots w_{j_k}^{\alpha_{m}} \right] f
$$
whenever $i_{1} < \cdots < i_{m}$ and $j_{1} < \cdots < j_{m}$,
and where
$[w^{\alpha}]f$ denotes the coefficient of~$w^{\alpha}$ in $f$. Denote
by $\QSym \subseteq  \Zzz[[w_{1},w_{2},\ldots]]$
the ring of quasisymmetric
functions.

For a composition 
$\alpha = (\alpha_{1}, \ldots, \alpha_{m})$
the {\em monomial quasisymmetric function}
$M_{\alpha}$ is given by
$$
M_{\alpha}
 =
\sum_{i_{1} < \cdots < i_{m}}
  w_{i_{1}}^{\alpha_{1}} \cdots w_{i_{m}}^{\alpha_{m}}.
$$
The monomial quasisymmetric functions $M_{\alpha}$
indexed by the compositions $\alpha$ in $\Sigma$ form a basis for the 
quasisymmetric functions.
A different basis
is given by the {\em fundamental quasisymmetric functions}
$L_{\alpha}$.
For fixed composition $\alpha$, 
the quasisymmetric function
$L_{\alpha}$  is defined by the sum
$$
L_{\alpha} = \sum_{\alpha \leq \beta} M_{\beta} .
$$
The quasisymmetric functions also form a Hopf algebra
where the coproduct is given by
$$  \Delta(M_{(\alpha_{1}, \ldots, \alpha_{m})})
  =
    \sum_{i=0}^{m}
           M_{(\alpha_{1}, \ldots, \alpha_{i})}
       \tensor
           M_{(\alpha_{i+1}, \ldots, \alpha_{m})}  . $$
A different way to view this coproduct is that
it is equivalent to the substitution
$$   \Delta(f(w_{1}, w_{2}, \ldots))
   =
     f(w_{1} \tensor 1, w_{2} \tensor 1, \ldots,
       1 \tensor w_{1}, 1 \tensor w_{2}, \ldots)   . $$
Malvenuto and Reutenauer~\cite{Malvenuto_Reutenauer}
defined an automorphism $\omega$ on
quasisymmetric functions by
the relation
$$    \omega(L_{\alpha}) = L_{\alpha^{c}}   . $$
The involution $\omega$ on $\QSym$ corresponds to
the involution $u \longmapsto \overline{u}$ in $\zab$.
The antipode on the Hopf algebra on quasisymmetric functions
is given by
$$    S(M_{\alpha})
   =
        (-1)^{n}
      \cdot
        \omega(M_{\alpha^{*}})  , $$
where $\alpha = (\alpha_{1}, \ldots, \alpha_{m})$
is a composition of $n$ and
$\alpha^{*}$ denotes the reverse composition
$(\alpha_{m}, \ldots, \alpha_{1})$.

For a sequence of labels
$\lambda = (\lambda_{1}, \lambda_{2}, \ldots, \lambda_{n})$
of length $n$, we define two compositions
$\rho^{R}(\lambda)$ and $\rho^{F}(\lambda)$ of $n$.
The composition
$\rho^{R}(\lambda)$ records
the rising runs in the sequence $\lambda$,
that is,
$\rho^{R}(\lambda) = (\rho_{1}, \rho_{2}, \ldots, \rho_{m})$
if
\begin{align*}
      \lambda_{1}
    \sim
      \cdots
    \sim
      \lambda_{\rho_{1}}
&    \not\sim
      \lambda_{\rho_{1} + 1}
    \sim
      \cdots
    \sim
      \lambda_{\rho_{1} + \rho_{2}} \\
& \not\sim
      \lambda_{\rho_{1} + \rho_{2} + 1}
    \sim
      \cdots
    \sim
      \lambda_{\rho_{1} + \cdots + \rho_{m-1}}
\\
& \not\sim
      \lambda_{\rho_{1} + \cdots + \rho_{m-1} + 1}
    \sim
      \cdots
    \sim
      \lambda_{n},  
\end{align*}
where $\sum_{i = 1}^m \rho_i = n$.
Similarly, let
$\rho^{F}(\lambda)$ record the falling runs in the sequence, that is,
if $\rho^{F} = (\rho_{1}, \rho_{2}, \ldots, \rho_{m})$
we have
\begin{align*}
      \lambda_{1}
    \not\sim
      \cdots
    \not\sim
      \lambda_{\rho_{1}}
&    \sim
      \lambda_{\rho_{1} + 1}
    \not\sim
      \cdots
    \not\sim
      \lambda_{\rho_{1} + \rho_{2}} \\
& \sim
      \lambda_{\rho_{1} + \rho_{2} + 1}
    \not\sim
      \cdots
    \not\sim
      \lambda_{\rho_{1} + \cdots + \rho_{m-1}}
\\
& \sim
      \lambda_{\rho_{1} + \cdots + \rho_{m-1} + 1}
    \not\sim
      \cdots
    \not\sim
      \lambda_{n} .
\end{align*}
Observe that in the poset $\Sigma_{n}$
the two compositions $\rho^{R}(\lambda)$ and $\rho^{F}(\lambda)$ are
complements of each other, that is,
$(\rho^{R}(\lambda))^{c} = \rho^{F}(\lambda)$.

For a bounded labeled digraph $G$ define the 
{\em rising} and {\em falling quasisymmetric
functions} 
by
\begin{equation}
\label{equation_rising_falling_quasisymmetric}
F^{R}(G)
   = 
\sum_{p} 
     L_{\rho^{R}(\lambda(p))}  
\:\:\:\: \text{ and } \:\:\:\:
F^{F}(G)
   = 
\sum_{p} 
     L_{\rho^{F}(\lambda(p))} , 
\end{equation}
where each sum is over all paths $p$ 
from $\hz$ to $\ho$ in the digraph $G$.
Since the two compositions
$\rho^{R}(\lambda)$ and $\rho^{F}(\lambda)$ are
complements, directly we have that the
two quasisymmetric functions are related by
the automorphism~$\omega$, that is,
$$ \omega(F^{R}(G))  =  F^{F}(G))  .  $$

Similar to the notion of the two polynomials
$\widetilde{r}_{x,y}(q)$ and
$\widetilde{f}_{x,y}(q)$, 
for $x \leq y$ 
define
the two polynomials
$\widetilde{R}_{x,y}(q)$ and
$\widetilde{F}_{x,y}(q)$ 
by
\begin{equation}
    \widetilde{R}_{x,y}(q)
   =
    \sum_{p} q^{\ell(p)}  
                                \:\:\:\: \text{ and } \:\:\:\:
    \widetilde{F}_{x,y}(q)
   =
    \sum_{p} q^{\ell(p)}   ,  
\label{equation_R_tilde_F_tilde}
\end{equation}
where the sum ranges over all rising,
respectively falling, paths from $x$ to $y$.
Directly we have the relations
$\widetilde{R}_{x,y}(q) = q \cdot \widetilde{r}_{x,y}(q)$
and
$\widetilde{F}_{x,y}(q) = q \cdot \widetilde{f}_{x,y}(q)$
for $x < y$.

\begin{proposition}
For a bounded labeled digraph $G$ the
two identities hold:
\begin{align}
F^{R}(G) \vrule_{w_{m+1} = w_{m+2} = \cdots = 0}
& =
\sum_{c} 
     \widetilde{R}_{x_{0},x_{1}}(w_{1})
   \cdot
     \widetilde{R}_{x_{1},x_{2}}(w_{2})
   \cdots
     \widetilde{R}_{x_{m-1},x_{m}}(w_{m})   , 
\label{equation_F_R} \\
F^{F}(G) \vrule_{w_{m+1} = w_{m+2} = \cdots = 0}
& =
\sum_{c} 
     \widetilde{F}_{x_{0},x_{1}}(w_{1})
   \cdot
     \widetilde{F}_{x_{1},x_{2}}(w_{2})
   \cdots
     \widetilde{F}_{x_{m-1},x_{m}}(w_{m})   ,
\label{equation_F_F} 
\end{align}
where each sum is over all multichains
$c = \{\hz = x_{0} \leq x_{1} \leq \cdots \leq x_{m} = \ho\}$
of length $m$ in the digraph~$G$.
\label{proposition_m}
\end{proposition}
\begin{proof}
Each side of 
equation~\eqref{equation_F_R}
is a polynomial in $w_{1}, w_{2}, \ldots, w_{n}$,
where $n$ is the length of the longest chain in the
digraph $G$.
Consider the coefficient of the monomial
$w_{i_{1}}^{\alpha_{1}} \cdot w_{i_{2}}^{\alpha_{2}}
                       \cdots w_{i_{m}}^{\alpha_{m}}$
on the right-hand side of equation~\eqref{equation_F_R},
where $\alpha = (\alpha_{1}, \alpha_{2}, \ldots, \alpha_{m})$
is a composition with $m \leq n$ and
$1 \leq i_{1} < i_{2} < \cdots < i_{m} \leq n$.
This counts the number of paths $p = (e_{1}, e_{2}, \ldots, e_{m})$
in the digraph $G$ such that
$\alpha_{1} + \alpha_{2} + \cdots + \alpha_{m} = n$
and
\begin{align*}
  \lambda(e_{1})
    \sim
 &
  \cdots
    \sim
  \lambda(e_{\alpha_{1}}) , \\
  \lambda(e_{\alpha_{1} + 1})
    \sim
 &
  \cdots
    \sim
  \lambda(e_{\alpha_{1} + \alpha_{2}}) , \\
 &
  \:\:\: \vdots 
 \\
  \lambda(e_{\alpha_{1} + \cdots + \alpha_{m-1} + 1})
    \sim
 &
  \cdots
    \sim
  \lambda(e_{\alpha_{1} + \cdots + \alpha_{m}}) , 
\end{align*}
and where the relation between
$\lambda(e_{\alpha_{1} + \cdots + \alpha_{i}})$
and
$\lambda(e_{\alpha_{1} + \cdots + \alpha_{i} + 1})$
is not known.
In other words, this coefficient enumerates the number of paths $p$
such that $\rho^{R}(\lambda(p)) \leq \alpha$.

The coefficient of
$w_{i_{1}}^{\alpha_{1}} \cdot w_{i_{2}}^{\alpha_{2}}
                       \cdots w_{i_{m}}^{\alpha_{m}}$
in the left-hand side of equation~\eqref{equation_F_R}
is the coefficient of
$M_{\alpha}$ in $F^{R}(G)$.  This coefficient is given by
\begin{align*}
[M_{\alpha}] F^{R}(G)
& =
[M_{\alpha}]
\sum_{p} 
     L_{\rho^{R}(\lambda(p))}   \\
& =
[M_{\alpha}]
\sum_{p} 
  \sum_{\rho^{R}(\lambda(p)) \leq \alpha}
     M_{\alpha} .
\end{align*}
This is the number of paths $p$ such that
$\rho^{R}(\lambda(p)) \leq \alpha$, proving
the first identity.
The second identity~\eqref{equation_F_F} follows
by a symmetric argument.
\end{proof}

Proposition~\ref{proposition_m}
can be reformulated as follows.
\begin{proposition}
For a bounded labeled digraph $G$ the two identities hold:
\begin{align}
F^{R}(G)
& =
\lim_{m \longrightarrow \infty}
\sum_{c} 
     \widetilde{R}_{x_{0},x_{1}}(w_{1})
   \cdot
     \widetilde{R}_{x_{1},x_{2}}(w_{2})
   \cdots
     \widetilde{R}_{x_{m-1},x_{m}}(w_{m})   , 
\label{equation_F_R_lim} \\
F^{F}(G)
& =
\lim_{m \longrightarrow \infty}
\sum_{c} 
     \widetilde{F}_{x_{0},x_{1}}(w_{1})
   \cdot
     \widetilde{F}_{x_{1},x_{2}}(w_{2})
   \cdots
     \widetilde{F}_{x_{m-1},x_{m}}(w_{m})   .
\label{equation_F_F_lim} 
\end{align}
\end{proposition}

Define the Cartesian product $G \times H$
of two digraphs $G$ and $H$ to be the digraph
with vertex set $V(G \times H) = V(G) \times V(H)$
and edge set
$E(G \times H) = V(G) \times E(H) \cup E(G) \times V(H)$,
where the edges are defined by
$\tail_{G \times H}((e,y)) = (\tail_{G}(e),y)$,
$\head_{G \times H}((e,y)) = (\head_{G}(e),y)$,
$\tail_{G \times H}((x,e)) = (x,\tail_{H}(e))$ and
$\head_{G \times H}((x,e)) = (x,\head_{H}(e))$.
Furthermore, for the Cartesian product of
labeled digraphs, set 
$\Lambda_{G \times H} = \Lambda_{G} \cup \Lambda_{H}$,
where the relation is defined by
$\lambda \sim \mu$
if and only if one of the following cases hold:
($i$)
$\lambda, \mu \in \Lambda_{G}, \lambda \sim_{\Lambda_{G}} \mu$,
($ii$)
$\lambda \in \Lambda_{G}, \mu \in \Lambda_{H}$,
($iii$)
$\lambda, \mu \in \Lambda_{H}, \lambda \sim_{\Lambda_{H}} \mu$.
Finally, the labels of the Cartesian product are defined by
$\lambda_{G \times H}((e,y)) = \lambda_{G}(e)$
and
$\lambda_{G \times H}((x,e)) = \lambda_{H}(e)$.

Observe that if both of the digraphs $G$ and $H$ are acyclic
then their Cartesian product is acyclic.
Similarly, if both digraphs are locally finite, then so is their product.

\begin{lemma}
For two labeled acyclic digraphs $G$ and $H$,
the $\widetilde{R}$- and $\widetilde{F}$-polynomials
of the Cartesian product $G \times H$ are given by
\begin{align*}
\widetilde{R}_{(x,z), (y,w)}(q)
& =
\widetilde{R}_{x, y}(q) \cdot \widetilde{R}_{z, w}(q) , \\
\widetilde{F}_{(x,z), (y,w)}(q)
& =
\widetilde{F}_{x, y}(q) \cdot \widetilde{F}_{z, w}(q) .
\end{align*}
\end{lemma}
\begin{proof}
A rising chain in $G \times H$
must first have labels from $\Lambda_{G}$
and then labels from $\Lambda_{H}$.
Thus the only way to have a rising chain in the interval
$[(x,z), (y,w)]$ is to first have a rising chain
in $[(x,z), (y,z)] \cong [x,y]$ and then a rising chain
in $[(y,z), (y,w)] \cong [z,w]$. Similarly, a falling
chain must have the labels from $\Lambda_{H}$ first
and then from $\Lambda_{G}$.
\end{proof}

\begin{proposition}
For two labeled acyclic digraphs $G$ and $H$,
the $F^{R}$ and $F^{F}$ quasisymmetric functions
of the interval $[(x,z), (y,w)])$
in the Cartesian product $G \times H$
are given by
\begin{align*}
F^{R}([(x,z), (y,w)])
& =
F^{R}([x,y]) \cdot F^{R}([z,w]) , \\
F^{F}([(x,z), (y,w)])
& =
F^{F}([x,y]) \cdot F^{F}([z,w]) .
\end{align*}
\label{proposition_algebra_map}
\end{proposition}
\begin{proof}
By equation~\eqref{equation_F_R_lim}
we have
\begin{align*}
F^{R}([(x,z), (y,w)])
& =
\lim_{m \longrightarrow \infty}
\sum
     \widetilde{R}_{(x_{0},z_{0}),(x_{1},z_{1})}(w_{1})
   \cdots 
     \widetilde{R}_{(x_{m-1},z_{m-1}),(x_{m},z_{m})}(w_{m})    \\
& =
\lim_{m \longrightarrow \infty}
\sum
     \widetilde{R}_{x_{0},x_{1}}(w_{1})
   \cdots
     \widetilde{R}_{x_{m-1},x_{m}}(w_{m})   \\
&
\hspace*{20 mm}
   \cdot
     \widetilde{R}_{z_{0},z_{1}}(w_{1})
   \cdots
     \widetilde{R}_{z_{m-1},z_{m}}(w_{m}) \\
& =
\lim_{m \longrightarrow \infty}
\left(
\sum
     \widetilde{R}_{x_{0},x_{1}}(w_{1})
   \cdots
     \widetilde{R}_{x_{m-1},x_{m}}(w_{m})  
\right) \\
&
\hspace*{15 mm}
   \cdot
\left(
\sum
     \widetilde{R}_{z_{0},z_{1}}(w_{1})
   \cdots
     \widetilde{R}_{z_{m-1},z_{m}}(w_{m}) 
\right) \\
& =
\left(
\lim_{m \longrightarrow \infty}
\sum
     \widetilde{R}_{x_{0},x_{1}}(w_{1})
   \cdots
     \widetilde{R}_{x_{m-1},x_{m}}(w_{m})  
\right) \\
&
   \cdot
\left(
\lim_{m \longrightarrow \infty}
\sum
     \widetilde{R}_{z_{0},z_{1}}(w_{1})
   \cdots
     \widetilde{R}_{z_{m-1},z_{m}}(w_{m}) 
\right) \\
& =
     F^{R}([x,y])
   \cdot
     F^{R}([z,w])    ,
\end{align*}
where the two first sums are over all
multichains
$(x,z) = (x_{0},z_{0}) \leq (x_{1},z_{1})
       \leq (x_{2},z_{2})
       \leq \cdots
       \leq (x_{m},z_{m}) = (y,w)$
and
the remaining sums are over the multichains
$x = x_{0} \leq x_{1} \leq x_{2} \leq \cdots
   \leq x_{m} = y$ and
$z = z_{0} \leq z_{1} \leq z_{2} \leq \cdots
   \leq z_{m} = w$.
The dual argument gives the second identity.
\end{proof}

\begin{proposition}
For $[x,y]$ an interval
in a labeled acyclic digraph $G$,
\begin{align*}
     \Delta(F^{R}([x,y]))
& =
     \sum_{x \leq z \leq y}
            F^{R}([x,z])) \tensor F^{R}([z,y]))  , \\
     \Delta(F^{F}([x,y]))
& =
     \sum_{x \leq z \leq y}
            F^{F}([x,z])) \tensor F^{F}([z,y]))  .
\end{align*}
\label{proposition_coalgebra_map}
\end{proposition}
\begin{proof}
Using equation~\eqref{equation_F_R_lim} we have
\begin{align*}
\Delta(F^{R}([x,y]))
& =
\lim_{m \longrightarrow \infty}
\sum_{c} 
     \widetilde{R}_{x_{0},x_{1}}(w_{1} \tensor 1)
   \cdots
     \widetilde{R}_{x_{m-1},x_{m}}(w_{m} \tensor 1)
   \\
&
\hspace*{20mm}
  \cdot
     \widetilde{R}_{x_{m},x_{m+1}}(1 \tensor w_{1})
   \cdots
     \widetilde{R}_{x_{2m-1},x_{2m}}(1 \tensor w_{m})   \\
& =
\lim_{m \longrightarrow \infty}
\sum_{x \leq z \leq y}
\left(
\sum_{c_{1}} 
     \widetilde{R}_{x_{0},x_{1}}(w_{1} \tensor 1)
   \cdots
     \widetilde{R}_{x_{m-1},x_{m}}(w_{m} \tensor 1)
\right) \\
&
\hspace*{20mm}
   \cdot
\left(
\sum_{c_{2}} 
     \widetilde{R}_{x_{m},x_{m+1}}(1 \tensor w_{1})
   \cdots
     \widetilde{R}_{x_{2m-1},x_{2m}}(1 \tensor w_{m})
\right)   \\
& =
\lim_{m \longrightarrow \infty}
\sum_{x \leq z \leq y}
\left(
\sum_{c_{1}} 
     \widetilde{R}_{x_{0},x_{1}}(w_{1})
   \cdots
     \widetilde{R}_{x_{m-1},x_{m}}(w_{m})
\right) \\
&
\hspace*{20mm}
   \tensor
\left(
\sum_{c_{2}} 
     \widetilde{R}_{x_{m},x_{m+1}}(w_{1})
   \cdots
     \widetilde{R}_{x_{2m-1},x_{2m}}(w_{m})
\right)   \\
& =
     \sum_{x \leq z \leq y}
            F^{R}([x,z])) \tensor F^{R}([z,y]))  ,
\end{align*}
where in the first sum the chain $c$ is
$\{x = x_{0} \leq x_{1} \leq \cdots \leq x_{m} \leq \cdots \leq x_{2m} = y\}$,
$z$ is the element $x_{m}$ in the chain $c$ and
the chains $c_{1}$ and $c_{2}$ are the two chains
$\{x = x_{0} \leq x_{1} \leq \cdots \leq x_{m} = z\}$,
respectively
$\{z = x_{m} \leq \cdots \leq x_{2m} = y\}$.
A symmetric argument gives the second identity.
\end{proof}

Let $\mathcal{H}$ be the linear span of 
bounded labeled acyclic digraphs.
The space $\mathcal{H}$ is a Hopf algebra
with the product given by the Cartesian product and the coproduct
given by
$$   \Delta(G)
   = 
     \sum_{\hz \leq z \leq \ho} [\hz,z] \tensor [z,\ho] . $$
We have the following corollary.
\begin{corollary}
The two quasisymmetric functions $F^{R}$ and $F^{F}$ are
Hopf algebra homomorphisms from~$\mathcal{H}$
to the quasisymmetric functions $\QSym$.
\end{corollary}
\begin{proof}
Follows directly from 
Proposition~\ref{proposition_coalgebra_map}.
\end{proof}

Generalizing~\cite[Lemma~5.1]{Ehrenborg_Readdy_Sheffer},
we have the following lemma.
\begin{lemma}
For a labeled acyclic graph $G$,
$$     \sum_{x \leq z \leq y}
            \widetilde{R}_{x,z}(q)
          \cdot
            \widetilde{F}_{z,y}(-q)
     =
       \delta_{x,y}   .  $$
\label{lemma_inverse}
\end{lemma}
\begin{proof}
Using the defining relation for the antipode $S$, we have that
\begin{align*}
\delta_{x,y}
& =
\sum_{x \leq z \leq y}
  F^{R}([x,z]) \cdot S(F^{R}([z,y])) \\
& =
\sum_{x \leq z \leq y}
    F^{R}([x,z])
  \cdot
    \left(
      \left.
        \omega(F^R([y,z]^{*}))
      \right|_{w_{1} = - w_{1}, w_{2} = - w_{2}, \ldots}
    \right) \\
& =
\sum_{x \leq z \leq y}
    F^{R}([x,z])
  \cdot
    \left(
      \left.
        F^F([y,z]^{*})
      \right|_{w_{1} = - w_{1}, w_{2} = - w_{2}, \ldots}
    \right) .
\end{align*}
Setting $w_{1} = q$ and $w_{2} = w_{3} = \cdots = 0$
the result follows by Proposition~\ref{proposition_m}.
\end{proof}

This lemma also has a direct bijective proof.

\begin{proof}[Second proof of Lemma~\ref{lemma_inverse}.]
Let ${\mathcal R}_{x,y}$
and ${\mathcal F}_{x,y}$
be the set of all rising, respectively falling,
paths from~$x$ to~$y$.
Consider the disjoint union
$$     {\mathcal U}_{x,y}
     =
       \bigcup_{x \leq z \leq y}
            {\mathcal R}_{x,z}
          \cdot
            {\mathcal F}_{z,y}    . $$
In other words, ${\mathcal U}_{x,y}$ is the set of all
pair of paths $(p_{1}, p_{2})$ such that 
$p_{1}$ is rising, $p_{2}$ is falling,
and $p_{1}$ ends where $p_{2}$ starts.
We would like to prove that
$$     \sum_{(p_{1}, p_{2}) \in {\mathcal U}_{x,y}}
          (-1)^{\ell(p_{2})} \cdot q^{\ell(p_{1}) + \ell(p_{2})}
    =
       \delta_{x,y}    .  $$
When $x = y$ the result is immediate.
We prove the case when $x < y$ by
a sign-reversing involution~$\sigma$.

Given a pair of paths $(p_{1}, p_{2})$ in
${\mathcal U}_{x,y}$
with
$p_{1} = (e_{1}, \ldots, e_{i})$,
$p_{2} = (f_{1}, \ldots, f_{j})$
and $i$ and $j$ not both equal to $0$,
define
another pair of paths $\sigma(p_{1},p_{2}) = (q_{1},q_{2})$
by the following four cases.
Case~(i):
if $i=0$, that is, $x=z$, let 
$q_{1} = (f_{1})$
and
$q_{2} = (f_{2}, \ldots, f_{j})$.
Case~(ii):
if $j=0$, that is, $z=y$, let 
$q_{1} = (e_{1}, \ldots, e_{i-1})$
and
$q_{2} = (e_{i})$.
Cases~(iii) and~(iv)
are both when $i$ and $j$ are greater than $0$.
Compare the two labels
$\lambda(e_{i})$ and $\lambda(f_{1})$
with the relation on $\Lambda$.
Case~(iii):
if $\lambda(e_{i}) \sim \lambda(f_{1})$
let the pair of paths be
$q_{1} = (e_{1}, \ldots, e_{i},f_{1})$
and
$q_{2} = (f_{2}, \ldots, f_{j})$.
Case~(iv):
otherwise, that is,
$\lambda(e_{i}) \not\sim \lambda(f_{1})$
let the pair of paths be
$q_{1} = (e_{1}, \ldots, e_{i-1})$
and
$q_{2} = (e_{i},f_{1}, \ldots, f_{j})$.

It is direct to verify that $\sigma$ is
an involution. 
Furthermore,
one has that
$\ell(p_{1}) + \ell(p_{2}) = \ell(q_{1}) + \ell(q_{2})$
and that
the lengths
of $p_{2}$ and $q_{2}$ have different parity.
Hence $\sigma$ is a sign-reversing involution,
proving the lemma.
\end{proof}

As corollary to Lemma~\ref{lemma_inverse}
we have the following result.
Compare with Exercise~5.11
in~\cite{Bjorner_Brenti}.
\begin{corollary}
For a balanced labeled acyclic graph $G$,
$$     \sum_{x \leq z \leq y}
            \widetilde{R}_{x,z}(q)
          \cdot
            \widetilde{R}_{z,y}(-q)
     =
       \delta_{x,y}   .  $$
For a bipartite balanced labeled acyclic graph $G$,
$$     \sum_{x \leq z \leq y}
            (-1)^{\ell(z,y)}
          \cdot
            \widetilde{R}_{x,z}(q)
          \cdot
            \widetilde{R}_{z,y}(-q)
     =
       \delta_{x,y}   .  $$
\end{corollary}

The quasi-symmetric functions
$F^{R}$ encode the same
information of the labeled digraph $G$
as the $\ab$-index $\Psi(G)$.
To make this more explicit, define the
linear map $\gamma : \zab \longrightarrow \QSym$
by
$$ \gamma\left(
     (\av-\bv)^{\alpha_{1}-1}
          \cdot
     \bv
          \cdot
     (\av-\bv)^{\alpha_{2}-1}
          \cdot
     \bv
          \cdots
     \bv
          \cdot
     (\av-\bv)^{\alpha_{k}-1}
         \right)
   =
     M_{\alpha}; $$
see~\cite[Section~3]{Ehrenborg_Readdy_Tchebyshev}.
The map $\gamma$ is a vector space isomorphism
between $\zab$ and the quasisymmetric function without
a constant term. Now we have
for a digraph $G$ the identity
$\gamma(\Psi(G)) = F^{R}(G)$.

Stembridge~\cite{Stembridge}
introduced a sub-Hopf algebra of
the quasisymmetric functions $\QSym$
known as the peak algebra $\Pi$.
It plays the same role as the subalgebra $\zcd$ of $\zab$.
Concretely, the peak algebra is the span of
the constant quasisymmetric function $1$
with
the image of $\zcd$ under the map $\gamma$.
Hence 
Theorem~\ref{theorem_cd} can be reformulated as
follows.
\begin{theorem}
For a labeled acyclic digraph $G$, the following
are equivalent:
\begin{itemize}
\item[(i)]
For every interval $[x,y]$ in the digraph $G$
and for every non-negative integer $k$,
the number of rising paths from $x$ to $y$ of length $k$
is equal to
the number of falling paths from $x$ to $y$ of length~$k$.

\item[(ii)]
For every interval $[x,y]$ in the digraph $G$
and for every {\em even} positive integer $k$,
the number of rising paths from $x$ to $y$ of length $k$
is equal to
the number of falling paths from $x$ to $y$ of length~$k$.

\item[(iii)]
The $F^{R}$ quasisymmetric function
of every interval $[x,y]$ in the digraph $G$
belongs to the peak algebra $\Pi$.
\end{itemize}
\end{theorem}

\section{Balanced linear edge labelings}
\label{section_linear} 

We call an edge labeling {\em linear} if the underlying relation
$(\Lambda, \sim)$ is that of a linear order.

\begin{theorem}
Let $u$ be a non-zero $\cd$-polynomial with non-negative
coefficients. Then there exists a bounded balanced
labeled acyclic digraph $G$
with linear edge labeling 
which satisfies $\Psi(G) = u$.
\label{theorem_linear}
\end{theorem}

In order to prove this theorem, we first need the following
two lemmas.
\begin{lemma}
Let $G_{1}$ and $G_{2}$ be two bounded digraphs
with balanced linear edge labeling.
Let the underlying label sets be $\Lambda_{1}$, respectively $\Lambda_{2}$.
Define a new bounded labeled digraph $H$ by
\begin{align*}
V(H) & = V(G_{1}) \cup V(G_{2}) , \\
E(H) & = E(G_{1}) \cup E(G_{2}) \cup \{h_{1},h_{2}\} 
\end{align*}
where the new edges are
$\tail(h_{i}) = \ho_{1}$
and
$\head(h_{i}) = \hz_{2}$.
Let the new label set be
$\Lambda  =  \Lambda_{1} \cup \Lambda_{2} \cup \{\mu_{1},\mu_{2}\}$
and the linear order be any shuffling of
$\Lambda_{1}$ and $\Lambda_{2}$ with the condition that
the new labels $\mu_{1}$ and $\mu_{2}$
are the minimal, respectively the maximal, element of
the linear order $\Lambda$.
Finally, let the labels of the new edges be
$\lambda(h_{i}) = \mu_{i}$.
Then the digraph $H$ has a balanced labeling which is linear,
and its $\cd$-index is given by
$$ \Psi(H) = \Psi(G_{1}) \cdot \dv \cdot \Psi(G_{2}) .  $$
\label{lemma_u_d_v}
\end{lemma}
\begin{proof}
Every path $p$
from $\hz_{G_{1}} = \hz_{H}$ 
to $\ho_{G_{2}} = \ho_{H}$
breaks into a path in $G_{1}$,
a path in $G_{2}$ and one of the new edges $h_{1}$ or $h_{2}$.
Observe that
$$   u(p) = u(p|_{G_{1}}) \cdot v \cdot u(p|_{G_{2}})  ,  $$
where $v = \bv\av$ if the new edge is $h_{1}$
and $v = \av\bv$ if the new edge is $h_{2}$.
Hence summing over all paths we have
$$ \Psi(H) = \Psi(G_{1}) \cdot (\bv\av + \av\bv) \cdot \Psi(G_{2}) .  $$
A similar argument shows that every interval of $H$ has a $\cd$-index
and hence the labeling is balanced.
\end{proof}

\begin{lemma}
Let $G_{1}$ and $G_{2}$ be two bounded digraphs
with balanced linear edge labelings.
Let $H$ be the bounded digraph obtained by the disjoint
union of $G_{1}$ and $G_{2}$
and identifying the minimal elements~$\hz_{G_{1}}$ and $\hz_{G_{2}}$,
and the maximal elements $\ho_{G_{1}}$ and $\ho_{G_{2}}$.
Then $H$ has a balanced linear edge labeling and
its $\cd$-index is the sum
$$   \Psi(H) = \Psi(G_{1}) + \Psi(G_{2})    .   $$
\label{lemma_u_plus_v}
\end{lemma}

\begin{proof}[Proof of Theorem~\ref{theorem_linear}.]
The strong Bruhat order on the dihedral group
is the butterfly poset and hence its
$\cd$-index is $\cv^{n}$.
Hence by Lemma~\ref{lemma_u_d_v}
for any $\cd$-monomial $v$
we can construct
a bounded labeled acyclic digraph $G$ with a balanced linear order
such that $\Psi(G) = v$.
By Lemma~\ref{lemma_u_plus_v}
this can be extended
to any non-negative $\cd$-polynomial.
\end{proof}

As a remark, 
Bayer and Hetyei's work on
the cone spanned
by Eulerian posets
shows one can obtain a limiting poset which
has $\cd$-index $2^k \cdot w$,
where $w$ is a $\cd$ monomial with exactly $k$ $\dv$'s;
see~\cite[Proposition 2.9]{Bayer_Hetyei}.

Theorem~\ref{theorem_linear} motivates us to 
make the following conjecture.
\begin{conjecture}
The $\cd$-index of a bounded labeled acyclic digraph $G$
with a balanced linear edge labeling is non-negative.
\label{conjecture_linear}
\end{conjecture}

This conjecture implies the non-negativity of the $\cd$-index
of Bruhat graphs;
see~\cite[Conjecture~6.1]{Billera_Brenti}.

\begin{conjecture}
[Billera--Brenti]
Let $(W,S)$ be a Coxeter system with $u,v \in W$ and
$u < v$.
Then the $\cd$-index of the interval
$[u,v]$ is non-negative.
\end{conjecture}

\section{Concluding remarks}

As was mentioned in the previous\
section,
verifying Conjecture~\ref{conjecture_linear}
would imply the non-negativity
of the $\cd$-index of Bruhat graphs~\cite{Billera_Brenti}.
In the case of Bruhat graphs of infinite Coxeter groups,
recent work of the 
Ehrenborg, Hetyei and Readdy~\cite{Ehrenborg_Hetyei_Readdy}
on level Eulerian posets, that is,
infinite Eulerian
posets with a local regularity condition,
may provide some insight.
Blanco~\cite{Blanco_shortest} has studied instances 
of non-negativity for the poset of shortest paths in the Bruhat
order.  
Shellability arguments have been used by Stanley
to prove 
the non-negativity of
the $\cd$-index of 
$S$-shellable spheres~\cite{Stanley_flag_vectors},
in Karu's argument for the non-negativity of the toric $g$-vector
of non-rational polytopes~\cite{Karu},
as well as Karu's  argument for the non-negativity of $\cd$-index of
Gorenstein$^{*}$ posets~\cite{Karu_cd}.

Recent work of Fan and He~\cite{Fan_He_2} have 
applied Karu's flip condition~\cite{Karu_complete}
to show the coefficient of $\dv c^i \dv \cv^j$ is non-negative
in the $\cd$-index of any Bruhat graph.
See also~\cite{Fan_He_1}.
Perhaps an analogous result
can be established
for digraphs having a balanced
linear order.

Another research
direction
is to develop Kazhdan--Lusztig polynomials for directed graphs.
In order to do this,
one must require the graph to be bipartite, just
as the bipartite condition holds for Bruhat graphs.
Brenti, Caselli and Marietti's
theory of special matchings of a Hasse diagram of a poset
parallels the notion of perfect matchings in the 
Bruhat graph~\cite{Brenti_Caselli_Marietti}.
Can this be extended to
balanced graphs?
Morel~\cite{Morel} has given a geometric
interpretation of
Brenti's~\cite{Brenti} non-recursive lattice path 
formulation of 
the Kazhdan--Lusztig polynomials
in the case of Weyl groups.
This may give some insight into Kazhdan--Lusztig
polynomials of directed graphs.

In Billera and Brenti's original paper, 
restricting the
$\cd$-index of Bruhat graphs to the highest degree terms
yields the $\cd$-index
of the Eulerian poset $[u,v]$.
Understanding the degree restricted $\cd$-index
may  suggest a natural decomposition of the
paths in the Bruhat graph.
For dihedral Coxeter systems Blanco
showed the $\cd$-index
is given by the Fibonacci polynomials~\cite{Blanco_dihedral}.
This gives evidence
for non-negativity of the $\cd$-index for Bruhat graphs.

Billera and Brenti's~\cite{Billera_Brenti} expression for 
the Kazhdan--Lusztig polynomial $P_{u,v}(q)$
in terms of the $\cd$-index $\Psi([u,v])$
was based upon 
showing the generalized Dehn--Sommerville relations
hold for
coefficients of polynomials 
arising in Kazhdan--Lusztig polynomials~\cite[Theorem 8.4]{Brenti},
quasisymmetric functions,
and the peak algebra.
Their expression is
\begin{equation}
  P_{u,v}(q) = \sum_{i=0}^{\lfloor n/2 \rfloor} a_i 
   \cdot q^i \cdot \mathcal{B}_{n-2i}(-q),
\label{equation_BB_identity}
\end{equation}
where
$\mathcal{B}_k(q)$ is the $k$th ballot polynomial
$
     \mathcal{B}_k(q) = 1/(k+1) \cdot 
             \sum_{i=0}^{\lfloor k/2 \rfloor} 
                        \binom{k+1}{i}  \cdot
                        (k+1-2i) \cdot q^i
$
and
the coefficients 
$a_i$ arise out of a nontrivial computation from the
coefficients of 
$\Psi([u,v])$.
Although a remarkable identity, it does not reveal why
$P_{u,v}(q)$ is non-negative for
Coxeter groups.

\noindent
As was noted in~\cite{Billera_Brenti}, restricting the
identity~(\ref{equation_BB_identity})
for the coefficients $a_i$
to the highest degree $\cd$-monomials
yields the
Bayer--Ehrenborg~\cite{Bayer_Ehrenborg} expression for the $g$-polynomial
of
$\Psi([u,v])^*$, implying the difference
$P_{u,v}(q) - g([u,v]^{*},q)$ is a function of lower
degree
$\cd$-coefficients; see~\cite[Remark~4.2]{Billera_Brenti}.
Brenti and Caselli~\cite{Brenti_Caselli}
 have a new identity for the Kazhdan--Lusztig 
polynomials in terms of signed polynomials 
arising from an index set of 
what they call ``slalom paths''
associated to a sparse subset $T \subseteq \{1, \ldots, n-1\}$.
This may be related to the Bayer--Ehrenborg $S$-diagram approach;
see~\cite[Section~7]{Bayer_Ehrenborg}.
Other  recent
papers describing the $\cd$-index as weighted
sum of lattice paths
include that of 
Slone~\cite{Slone}
for the $\cd$-index of the mixing operator,
and B.\ Fox~\cite{Fox_a_lattice_path} for the 
the $\cd$-index of the diamond product of two Eulerian posets.
On the polytope level, these 
poset operations correspond to taking 
the join of polytopes 
and the Cartesian product of polytopes,
respectively.
We expect these ideas to be fruitful
for developing 
the Kazhdan--Lusztig polynomial of certain balanced graphs.

Reading~\cite{Reading} provided a recursive method to compute
$\cd$-index of any interval in the Bruhat order.
Can his methods be extended to any interval in the Bruhat graph?
Do they generalize to balanced graphs?

Returning our attention to graded posets and especially Eulerian posets,
when do these posets possess a labeling?
There are Eulerian posets which do not have
an $R$-labeling; see~\cite{Ehrenborg_Readdy_non-R}.
What more can be said about Eulerian posets that have
an $R$-labeling? 
Conjecture~\ref{conjecture_linear} implies 
that Eulerian posets with an $R$-labeling have a 
non-negative $\cd$-index.

\section*{Acknowledgements}

The first author was partially supported by
National Science Foundation grant 0902063.
This work was partially supported by a grant from the 
Simons Foundation (\#429370 to Richard Ehrenborg;
\#206001 and \#422467 to Margaret Readdy).
Both authors would like to thank the Princeton University Mathematics
Department for its hospitality and support during
the academic year 2014--2015,
and the Institute for Advanced Study for hosting a 
research visit in Summer 2019.

\newcommand{\journal}[6]{{\sc #1,} #2, {\it #3} {\bf #4} (#5), #6.}
\newcommand{\book}[4]{{\sc #1,} ``#2,'' #3, #4.}
\newcommand{\bookf}[5]{{\sc #1,} ``#2,'' #3, #4, #5.}
\newcommand{\arxiv}[3]{{\sc #1,} #2, {\tt #3}.}
\newcommand{\preprint}[3]{{\sc #1,} #2, {(#3)}.}
\newcommand{\thesis}[4]{{\sc #1,} ``#2,'' Doctoral dissertation, #3, #4.}

\end{document}